\documentclass[12pt]{amsart}
\usepackage{amsmath,amsfonts,amssymb,amsthm,a4wide,tensor,url,mathscinet,mathdots}

\newtheorem{theorem}{Theorem}[section]
\newtheorem{lemma}[theorem]{Lemma}
\newtheorem{definition}[theorem]{Definition}
\newtheorem{proposition}[theorem]{Proposition}
\newtheorem{corollary}[theorem]{Corollary}
\newtheorem{conjecture}[theorem]{Conjecture}
\theoremstyle{remark}
\newtheorem{remark}[theorem]{Remark}

\newcommand{\N}{\mathbb{N}}
\newcommand{\Z}{\mathbb{Z}}
\newcommand{\R}{\mathbb{R}}
\newcommand{\C}{\mathbb{C}}
\newcommand{\Cat}{\mathcal{C}}

\newcommand{\one}{\mathbf{1}}
\newcommand{\freegp}[1]{\Z({#1})}
\newcommand{\MS}[1]{\N({#1})}
\newcommand{\red}[1]{#1\operatorname{-red}}
\newcommand{\Galinv}{\varrho}
\newcommand{\Isom}{\operatorname{Isom}}
\newcommand{\Cuspdata}{\mathcal{D}_{\operatorname{cusp}}}
\newcommand{\Seg}{\mathcal{SEG}}
\newcommand{\lextend}[1]{\tensor*[^+]#1{}}
\newcommand{\lshrink}[1]{\tensor*[^-]#1{}}
\newcommand{\Zpstv}{Z-positive}
\newcommand{\data}{\mathfrak{d}}
\newcommand{\prcspdata}{p}
\newcommand{\mnml}[1]{{#1}-confined}
\newcommand{\prcspdataGr}{\bar p}

\newcommand{\RepsGL}{\mathcal{C}^{\GL}}
\newcommand{\RepsG}{\mathcal{C}^G}
\newcommand{\id}{\operatorname{Id}}
\newcommand{\Gr}{\mathcal{R}}
\newcommand{\GrGL}{\Gr^{\GL}}
\newcommand{\GrG}{\Gr^G}
\newcommand{\JacG}{J^G}
\newcommand{\JacGL}{J^{\GL}_{\max}}
\newcommand{\JacGLs}{J^{\GL}_{\operatorname{submax}}}
\newcommand{\JacGLnm}{J^{\GL}}
\newcommand{\comult}{\operatorname{m}^*}
\newcommand{\comults}{\operatorname{m}^*_{\operatorname{it}}}
\newcommand{\comod}{M^*}
\newcommand{\comodmax}{M^*_{\max}}
\newcommand{\comodrev}{M^*_{\operatorname{ex}}}
\newcommand{\comultG}{\mu^*}
\newcommand{\lnrset}{\mathfrak{S}^l}
\newcommand{\rnrset}{\mathfrak{S}^r}
\newcommand{\lft}{^{\operatorname{l}}}
\newcommand{\rgt}{^{\operatorname{r}}}

\newcommand{\JH}{\operatorname{JH}}
\newcommand{\ex}{\operatorname{ex}}
\newcommand{\pstv}[1]{#1^\uparrow}
\newcommand{\abs}[1]{\left|{#1}\right|}
\newcommand{\supp}{\operatorname{supp}}
\newcommand{\GL}{\operatorname{GL}}
\newcommand{\Irr}{\operatorname{Irr}}
\newcommand{\IrrGL}{\Irr\GL}
\newcommand{\IrrgenGL}{\Irr_{\operatorname{gen}}\GL}
\newcommand{\IrrcogenGL}{\Irr_{\operatorname{cogen}}\GL}
\newcommand{\IrrG}{\Irr G}
\newcommand{\IrrctmpG}{\Irr_{\operatorname{cotemp}}G}
\newcommand{\IrrtmpG}{\Irr_{\operatorname{temp}}G}
\newcommand{\dunion}{\sqcup}
\newcommand{\cuspred}{S}
\newcommand{\Cusp}{\operatorname{Cusp}}
\newcommand{\CuspGL}{\Cusp^{\GL}}
\newcommand{\CuspG}{\Cusp^G}
\newcommand{\m}{\mathfrak{m}}
\newcommand{\n}{\mathfrak{n}}
\newcommand{\soc}{\operatorname{soc}}
\renewcommand{\cos}{\operatorname{cos}}
\newcommand{\std}[1]{\mathfrak{z}(#1)}
\newcommand{\cstd}[1]{\mathfrak{l}(#1)}
\newcommand{\SI}{SI}
\newcommand{\CSI}{CSI}
\newcommand{\lshft}[1]{\overset{\leftarrow}{#1}}
\newcommand{\rshft}[1]{\overset{\rightarrow}{#1}}
\newcommand{\sdp}{\rtimes}
\newcommand{\expo}{\mathfrak{c}}

\newcommand{\Reps}{\mathcal{C}}

\numberwithin{equation}{section}

\begin{document}

\title{Some results on reducibility of parabolic induction for classical groups}
\author{Erez Lapid}
\address{Department of Mathematics, Weizmann Institute of Science, Rehovot 7610001, Israel}
\email{erez.m.lapid@gmail.com}
\author{Marko Tadi\'c}
\address{Department of Mathematics, University of Zagreb, Bijenicka 30, Zagreb 10000, Croatia}
\email{tadic@math.hr}
\date{\today}

\maketitle

\begin{abstract}
Given a (complex, smooth) irreducible representation $\pi$ of the general linear group over a non-archimedean local field and an irreducible supercuspidal representation
$\sigma$ of a classical group, we show that the (normalized) parabolic induction $\pi\rtimes\sigma$ is reducible if there exists $\rho$ in the
supercuspidal support of $\pi$ such that $\rho\rtimes\sigma$ is reducible. In special cases we also give irreducibility criteria for $\pi\sdp\sigma$
when the above condition is not satisfied.
\end{abstract}

\setcounter{tocdepth}{1}
\tableofcontents

\section{Introduction}
In this paper we study reducibility of parabolic induction of representations of classical groups over $p$-adic fields.
(All representations considered are complex, smooth and of finite length, hence admissible.)
Of course, the problem goes back to the early days of representation theory. It is important not only in its own right
but also for studying other classes of representations such as discrete series \cite{MR1896238} and unitary representations \cite{MR1181278}.

We fix a classical group $G$ over a non-archimedean local field $F$ of characteristic $0$ and a supercuspidal irreducible representation $\sigma$
of $G(F)$. In the unitary case let $E/F$ be the quadratic extension over which $G$ splits and let $\Galinv$ be the Galois involution.
In the other cases let $E=F$ and $\Galinv=\id$.

Denote by $\tilde{}$ the composition of $\Galinv$ with the involution $g\mapsto\tensor[^t]g{^{-1}}$ of $\GL_n(E)$.
(We continue to denote by $\tilde{}$ the induced involution on the various objects pertaining to $\GL_n(E)$.)
Denote by $\IrrGL$ the set of irreducible representations (up to equivalence) of any of the groups $\GL_n(E)$, $n\ge0$.
Let $\CuspGL\subset\IrrGL$ be the set of (not necessarily unitarizable) irreducible supercuspidal representations (up to equivalence)
of any of the groups $\GL_n(E)$, $n\ge1$. For any $\rho\in\CuspGL$ and $x\in\R$ let $\rho[x]\in\CuspGL$ be the twist of $\rho$ by the character $\abs{\det\cdot}^x$.
Let $\rho[\Z]=\{\rho[m]:m\in\Z\}$ and $\rho[\R]=\{\rho[x]:x\in\R\}$.
For any $\pi\in\IrrGL$ we denote by $\supp\pi\subset\CuspGL$ the supercuspidal support of $\pi$.

A major role is played by the set
\[
\cuspred_\sigma=\{\rho\in\CuspGL:\rho\sdp\sigma\text{ is reducible}\}.
\]
(As usual, we denote by $\sdp$ (resp., $\times$) normalized parabolic induction for classical groups (resp., the general linear group).)
By standard results, $\tilde\cuspred_\sigma=\cuspred_\sigma$ and for any $\rho\in\cuspred_\sigma$ we have
$\rho[\R]\cap\cuspred_\sigma=\{\rho,\tilde\rho\}$.
A deeper result due to M\oe glin states that in fact $\tilde\rho\in\rho[\Z]$.

Let $\pi\in\IrrGL$. The question of reducibility of $\pi\sdp\sigma$ naturally divides into two cases, according to whether or not
$\supp\pi$ intersects $\cuspred_\sigma$.
In the first case there is a clear-cut answer, which is our first main result.
\begin{theorem} \label{thm: red}
Let $\pi\in\IrrGL$ be such that $\supp\pi\cap\cuspred_\sigma\ne\emptyset$. Then $\pi\sdp\sigma$ is reducible.
\end{theorem}

In the case where $\supp\pi\cap\cuspred_\sigma=\emptyset$ things are more complicated and we do not have a general simple characterization
of the irreducibility of $\pi\sdp\sigma$. We need to impose a condition on either $\pi$ or $\sigma$.
Let us describe our partial results in this direction.
We first remark that it is easy to reduce the question (assuming irreducibility of parabolic induction for the general linear group is understood)
to the case where $\supp\pi\subset\rho[\Z]$ for some $\rho\in\CuspGL$ such that
$\tilde\rho\in\rho[\Z]$. Assume further that $\rho\in\cuspred_\sigma$.
(We remark that the set $\cup_{\rho\in\cuspred_\sigma}\rho[\Z]$ is independent of $\sigma$, and in fact depends only on the type of $G$.)

For any $\pi\in\IrrGL$ we construct an auxiliary representation $\pi_+\in\IrrGL$. (See Definition \ref{def: pi+}.)
Namely, if $\pi$ corresponds to a multisegment $\m$ under the Zelevinsky classification, then
$\pi_+$ corresponds to the submultisegment of $\m$ consisting of the segments with positive exponents.\footnote{By applying
the Zelevinsky--Aubert involution, we can work with the Langlands classification instead.}

We also recall the notion of ladder representations \cite{MR3163355}, which are generalization of the Speh representations.

Our second main result is the following.
\begin{theorem} \label{thm: mainintred}
Let $\pi\in\IrrGL$ be such that $\supp\pi\cap\cuspred_\sigma=\emptyset$.
Assume that
\[
\text{$\pi$ is a ladder representation}
\]
or that
\[
\supp\pi\subset\rho[\Z]\text{ where }\rho\in\cuspred_\sigma\text{ and }\rho\in\{\tilde\rho,\tilde\rho[\pm1],\tilde\rho[\pm2]\}.
\]
Then $\pi\sdp\sigma$ is irreducible if and only if $\pi_+\times\tilde\pi_+$ is irreducible (where $\tilde\pi_+=(\tilde\pi)_+$).
\end{theorem}


As a supplement to the theorem, we remark that in the case where $\pi_1,\pi_2\in\IrrGL$ are two ladder representations,
the irreducibility of $\pi_1\times\pi_2$ can be characterized explicitly \cite[Proposition 6.20]{MR3573961}.
We also remark that by the work of Shahidi \cite{MR1070599}, the condition $\rho\in\{\tilde\rho,\tilde\rho[\pm1],\tilde\rho[\pm2]\}$ is satisfied for any
$\rho\in\cuspred_\sigma$
if $G_n$ is quasi-split and $\sigma$ admits a Whittaker model with respect to a non-degenerate character of a maximal unipotent subgroup of $G_n$.
(Without restriction on $G_n$ or $\sigma$, the condition $\rho\in\{\tilde\rho,\tilde\rho[\pm1]\}$ is satisfied for all but finitely many
$\rho\in\cuspred_\sigma$.)

On the other hand, if $\rho=\tilde\rho[m]\in\cuspred_\sigma$ with $m>2$ then it is easy to construct $\pi\in\IrrGL$ such that $\pi=\pi_+$,
$\supp\pi\subset\rho[\Z]\setminus\cuspred_\sigma$ and $\pi\sdp\sigma$ is reducible.
Thus, the additional conditions on $\pi$ or $\sigma$ in Theorem \ref{thm: mainintred} are not superfluous.
We do not know whether the irreducibility of
$\pi_+\times\tilde\pi_+$ is a necessary condition for the irreducibility of $\pi\sdp\sigma$ in general.
At any rate we expect the following to hold.

\begin{conjecture}
Assume that $\pi\in\IrrGL$ is such that $\supp\pi\cap\cuspred_\sigma=\emptyset$.
Then the irreducibility of $\pi\sdp\sigma$ depends only on $\pi$ and not on $\sigma$.
\end{conjecture}
Granted this conjecture, an interesting natural follow-up question would be to characterize those $\pi$'s such that
$\pi\sdp\sigma$ is irreducible for any supercuspidal $\sigma$ such that $\supp\pi\cap \cuspred_\sigma=\emptyset$.
This is already non-trivial if $\pi=\pi_+$.

The structure of the paper is as follows.
In \S\ref{sec: GLn} we introduce the notation and recollect some basic results pertaining to the representation theory of the general linear group,
in the spirit of Bernstein--Zelevinsky.
In \S\ref{sec: classical} we do the same for classical groups.
Theorem \ref{thm: red} is proved in \S\ref{sec: red} by reducing it to a simple case which can be treated by hand.
In \S\ref{sec: twosegments} we treat the case of two segments. Although it is not logically necessary for the general case,
we include it in order to illustrate the idea.
Finally in \S\ref{sec: irred} we prove Theorem \ref{thm: mainintred} as well as some other basic facts about irreducibility.

Our results generalize known results such as some of the results in
\cite{MR631958, MR1134591, MR1228612, MR1346929, MR1420710, MR1665057, MR2017065, MR1658535, MR2504024}.
However, our approach is somewhat different.

A natural next step would be to extend our results beyond the case where $\sigma$ is supercuspidal.

Part of this work was done while the authors were hosted by the Mathematisches Forschungsinstitut Oberwolfach for a
``Research in Pairs'' project.
We are very grateful to the MFO for providing ideal working conditions.
Finally, it is a pleasure to thank David Goldberg, Max Gurevich, Marcela Hanzer, Chris Jantzen, Alberto M\'inguez, Colette M\oe glin and Goran Muic for useful correspondence.

\subsection{General notation}

Throughout we fix a non-archimedean local field $F$ with normalized absolute value $\abs{\cdot}$.

\emph{From section \ref{sec: classical} onward we assume that $F$ is of characteristic $0$.}
(Hopefully this assumption can be lifted.)

For any set $X$, $\freegp{X}$ (resp., $\MS{X}$) denotes the free abelian group (resp., monoid) generated by $X$.
We denote the resulting order on $\freegp{X}$ by $\le$.
We can think of an element of $\MS{X}$ as a finite multiset consisting of elements of $X$.
If $a=x_1+\dots+x_k\in\MS{X}$ with $x_1,\dots,x_k\in X$ we call the underlying set $\{x_1,\dots,x_k\}$ the support of $a$.

All abelian categories considered in this paper are essentially small, $\C$-linear and locally finite.
(See \cite[Ch.~1]{MR3242743} for basic terminology and more details.)
For any such category $\Cat$ let $\Irr\Cat$ be the set of isomorphism classes of simple objects of $\Cat$ and $\Gr(\Cat)$ the Grothendieck group of $\Cat$,
isomorphic to $\freegp{\Irr\Cat}$, and thus an ordered group.
For any object $\pi$ of $\Cat$ we denote by $\JH(\pi)$ the Jordan--H\"older sequence of $\pi$ considered as an element of $\MS{\Irr\Cat}$
and by $\ell(\pi)$ the length of $\pi$ (i.e., the number of elements of $\JH(\pi)$, counted with multiplicities).
For simplicity, we often write $\pi_1\le\pi_2$ if $\JH(\pi_1)\le\JH(\pi_2)$.
Denote by $\soc(\pi)$ (resp., $\cos(\pi)$) the socle (resp., cosocle) of an object $\pi$ of $\Cat$, i.e.,
the largest semisimple subobject (resp., quotient) of $\pi$.
We say that $\pi$ is \SI\ (socle irreducible) (resp., \CSI; cosocle irreducible) if $\soc(\pi)$ (resp., $\cos(\pi)$) is simple and occurs with multiplicity one in $\JH(\sigma)$.
Obviously,
\begin{equation} \label{eq: obvious}
\text{if $\pi$ is \SI\ then so is any non-zero subobject $\sigma$ of $\pi$, and $\soc(\sigma)=\soc(\pi)$}.
\end{equation}
Clearly, $\pi$ is simple if and only if it is \SI\ and $\soc(\pi)\simeq\cos(\pi)$.
If $\Cat$ admits a duality functor $^\vee$ (as will be the case throughout the paper) then $\cos(\pi)^\vee=\soc(\pi^\vee)$. Thus,
\begin{equation} \label{eq: pisimple}
\text{$\pi$ is simple if and only if $\pi$ or $\pi^\vee$ is \SI\ and $\soc(\pi)^\vee\simeq\soc(\pi^\vee)$}.
\end{equation}

Now let $G$ be a connected reductive group over $F$.\footnote{Later on we will also consider orthogonal groups, which are not connected}
We denote by $\Reps(G)$ the category of admissible, finitely generated representations of $G(F)$ over $\C$.
We refer to \cite{MR2567785} for standard facts about representation theory of $G(F)$, some of which we will freely use below.
We denote the contragredient duality functor by $^\vee$.
For simplicity we write $\Irr G=\Irr\Reps(G)$ and $\Gr(G)=\Gr(\Reps(G))$. 
In particular, for the trivial group, $\Irr 1$ consists of a single element which we denote by $\one$.
Note that if $G_1$ and $G_2$ are reductive groups over $F$ then $\Reps(G_1\times G_2)$ is the tensor product of $\Reps(G_1)\otimes\Reps(G_2)$
in the sense of \cite[\S5]{MR1106898}.
(This follows from [ibid., Proposition 5.3] and the fact that the tensor product of categories commutes with inductive limits [ibid., p. 143].)

Denote by $\Cusp G\subset\Irr G$ the set of irreducible supercuspidal representations of $G(F)$ (up to equivalence).
Let $\Cuspdata(G)$ be the set of cuspidal data of $G$, i.e., $G(F)$-conjugacy classes of pairs $(M,\sigma)$ consisting of
a Levi subgroup $M$ of $G$ defined over $F$ and $\sigma\in\Cusp M$.
For any $\data\in\Cuspdata(G)$ let $\Reps(G)_{\data}$ be the Serre subcategory\footnote{See \cite[Definition 4.14.1]{MR3242743}}
of $\Reps(G)$ generated by the (normalized) parabolic induction
$\operatorname{Ind}_{P(F)}^{G(F)}\sigma$ where $(M,\sigma)\in\data$ and $P$ is a parabolic subgroup of $G$ defined over $F$ with Levi subgroup $M$.
(The definition is independent of the choice of $(M,\sigma)$ and $P$. We also recall that parabolic induction commutes with $^\vee$.)
The category $\Reps(G)$ splits as
\[
\Reps(G)=\oplus_{\data\in\Cuspdata(G)}\Reps(G)_{\data}
\]
where for each $\data$, $\Irr\Reps(G)_{\data}$ is finite. Accordingly,
\begin{equation} \label{eq: Grcd}
\Gr(G)=\oplus_{\data\in\Cuspdata(G)}\Gr(G)_{\data}\ \ \text{where }\Gr(G)_{\data}=\Gr(\Reps(G)_{\data}).
\end{equation}
We denote by $\prcspdata_{\data}$ the projection $\Reps(G)\rightarrow\Reps(G)_{\data}$.
More generally, if $A\subset\Cuspdata(G)$ then we denote the projection $\Reps(G)\rightarrow\Reps(G)_A:=\oplus_{\data\in A}\Reps(G)_\data$ by $\prcspdata_A$.
We write $\prcspdataGr_A:\Gr(G)\rightarrow\Gr(G)_A=\oplus_{\data\in A}\Gr(G)_{\data}$ for the corresponding projection in the Grotehndieck group.

We denote by $^t$ the Zelevinsky--Aubert involution on $\Gr(G)$. It respects the decomposition \eqref{eq: Grcd}
and induces an involution, also denoted by $^t$ on $\Irr G$ \cite{MR1285969, MR1390967} (cf.~\cite[\S6]{MR2113853}
for the even orthogonal case).\footnote{See \cite{1701.07329} for a recent approach which highlights the functorial properties of $^t$}

\section{The general linear group} \label{sec: GLn}
In this section we recall some facts about representation theory of the general linear group.
Most of the results are standard and go back to the seminal work of Bernstein--Zelevinsky.
\subsection{Notation}
Let
\[
\RepsGL=\oplus_{n\ge0}\Reps(\GL_n),\ \ \IrrGL=\Irr\RepsGL=\dunion_{n\ge0}\Irr\GL_n,\ \ \GrGL=\Gr(\RepsGL)=\oplus_{n\ge0}\Gr(\GL_n).
\]
If $\pi\in\Reps(\GL_n)$, we write $\deg\pi=n$,
and for any $x\in\R$ we denote by $\pi[x]$ the representation obtained from $\pi$ by twisting by the character $\abs{\det}^x$.
In particular set $\rshft{\pi}=\pi[1]$ and $\lshft{\pi}=\pi[-1]$.
If $\pi\in\Irr(\GL_n)$, $n>0$ (or more generally, if $\pi$ has a central character $\omega_\pi$) let $\expo(\pi)$ be the real number such that $\pi[-\expo(\pi)]$
has a unitary central character. (That is, the character $\omega_\pi\abs{\cdot}^{-\deg\pi\cdot\expo(\pi)}$ is unitary.)
We also set $\expo(\one)=0$.
Note that $\expo(\pi^\vee)=-\expo(\pi)$ and
\[
\expo(\pi_1\times\pi_2)=(\expo(\pi_1)\deg\pi_1+\expo(\pi_2)\deg\pi_2)/(\deg\pi_1+\deg\pi_2).
\]

For any $n,m\ge0$ let
\[
\times:\Reps(\GL_n)\times\Reps(\GL_m)\rightarrow\Reps(\GL_{n+m})
\]
be the bilinear biexact bifunctor of (normalized) parabolic induction with respect to the parabolic subgroup of block upper triangular matrices.
We also denote by $\times$ the resulting bilinear biexact bifunctor
\[
\times:\RepsGL\times\RepsGL\rightarrow\RepsGL.
\]
Together with the unit element $\one$ and the isomorphism of induction by stages, this endows $\RepsGL$
with the structure of a monoidal category (and hence, a ring category over $\C$ in the sense of \cite[Definition 4.2.3]{MR3242743}).
This structure also induces a $\Z_{\ge0}$-graded ring structure on $\GrGL$. Although $\times$ is not symmetric, $\GrGL$ is a commutative ring.

Let
\begin{gather*}
\JacGL:\RepsGL\rightarrow\RepsGL\otimes\RepsGL=\oplus_{n,m}\Reps(\GL_n\times\GL_m),\\
\JacGLnm_{(n,m)}:\Reps(\GL_{n+m})\rightarrow\Reps(\GL_n)\otimes\Reps(\GL_m)=\Reps(\GL_n\times\GL_m),\ \ n,m\ge0
\end{gather*}
be the left-adjoint functors to $\times$. Thus, $\JacGL=\oplus_n(\oplus_{n_1+n_2=n}\JacGLnm_{(n_1,n_2)})$ and $\JacGLnm_{(n,m)}$ is the (normalized) Jacquet functor.
We denote by $\comult$ the resulting ring homomorphism\footnote{The fact that $\comult$ is a homomorphism follows from the geometric lemma of Bernstein--Zelevinsky.}
\[
\comult:\GrGL\rightarrow\GrGL\otimes\GrGL.
\]

We also write
\[
\JacGLs=(\JacGL\otimes\id)\circ\JacGL=(\id\otimes\JacGL)\circ\JacGL:\RepsGL\rightarrow\RepsGL\otimes\RepsGL\otimes\RepsGL
\]
and
\[
\comults=(\comult\otimes\id)\circ\comult=(\id\otimes\comult)\circ\comult:\GrGL\rightarrow\GrGL\otimes\GrGL\otimes\GrGL.
\]


Let $\CuspGL=\dunion_{n>0}\Cusp\GL_n$. (Note that we exclude $\one$ from $\CuspGL$.)
We identify $\dunion_{n\ge0}\Cuspdata(\GL_n)$ with $\MS{\CuspGL}$.
Thus, we get a decomposition
\[
\RepsGL=\oplus_{\data\in\MS{\CuspGL}}\RepsGL_{\data}
\]
and for any subset $X\subset\MS{\CuspGL}$ a projection
\[
\prcspdata_X:\RepsGL\rightarrow\RepsGL_X:=\oplus_{\data\in X}\RepsGL_{\data}.
\]
We write $\GrGL_X=\Gr(\RepsGL_X)=\oplus_{\data\in X}\Gr(\RepsGL_{\data})$ and $\prcspdataGr_X:\GrGL\rightarrow\GrGL_X$.
For any $\data\in\MS{\CuspGL}$ we have $\RepsGL_{\data}\subset\Reps(\GL_{\deg\data})$ where we extend $\deg$ to $\MS{\CuspGL}$ by linearity.
For any $\data_1,\data_2\in\MS{\CuspGL}$ we have $\RepsGL_{\data_1}\times\RepsGL_{\data_2}\subset\RepsGL_{\data_1+\data_2}$.
If $0\ne\pi\in\RepsGL_{\data}$ we write $\data_\pi=\data\in\MS{\CuspGL}$ and denote by $\supp\pi\subset\CuspGL$ the support of $\data_\pi$.

For any $X\subset\MS{\CuspGL}$ we write $\JacGLnm_{X;*}$ for the composition of $\JacGL$ with
\[
\RepsGL\otimes\RepsGL\xrightarrow{\prcspdata_X\otimes\id}\RepsGL_X\otimes\RepsGL.
\]
Analogously for $\JacGLnm_{*;X}$. We also write $\comult_{X;*}$ and $\comult_{*;X}$ for the corresponding homomorphisms
\[
\GrGL\rightarrow\GrGL_X\otimes\GrGL,\ \ \GrGL\rightarrow\GrGL\otimes\GrGL_X.
\]

Similarly, given $X,Y\subset\MS{\CuspGL}$ we write $\JacGLnm_{X;*;Y}$ for the composition
\[
\RepsGL\xrightarrow{\JacGLs}\RepsGL\otimes\RepsGL\otimes\RepsGL\xrightarrow{\prcspdata_X\otimes\id\otimes\prcspdata_Y}
\RepsGL_X\otimes\RepsGL\otimes\RepsGL_Y
\]
and
\[
\comult_{X;*;Y}:\GrGL\rightarrow\GrGL_X\otimes\GrGL\otimes\GrGL_Y
\]
for the corresponding map of Grothendieck groups.

\subsection{Derivatives}\footnote{This notion should not be confused with Zelevinsky's notion of derivative.} \label{sec: derivatives}
\begin{definition}
Let $\rho\in\CuspGL$. We say that $\pi\in\IrrGL$ is left $\rho$-reduced if the following equivalent conditions are satisfied.
\begin{enumerate}
\item There does not exist $\pi'\in\IrrGL$ such that $\pi\hookrightarrow\rho\times\pi'$.
\item $\JacGLnm_{\{\rho\};*}(\pi)=0$, i.e., there does not exist $\pi'\in\IrrGL$ such that $\rho\otimes\pi'\le\JacGL(\pi)$.
\item $\JacGLnm_{\MS{\{\rho\}};*}(\pi)=\one\otimes\pi$, i.e., there do not exist $\pi_1,\pi_2\in\IrrGL$ such that
$\pi_1\otimes\pi_2\le\JacGL(\pi)$ and $\supp\pi_1=\{\rho\}$.
\end{enumerate}
Given $A\subset\CuspGL$ we say that $\pi\in\IrrGL$ is left $A$-reduced if $\pi$ is left $\rho$-reduced for any $\rho\in\CuspGL$.
Equivalently, $\JacGLnm_{\MS{A};*}(\pi)=\one\otimes\pi$.
Similarly for right $\rho$-reduced and right $A$-reduced representations.
\end{definition}

We note that $\pi$ is left $\rho$-reduced if and only if $\pi^\vee$ is right $\rho^\vee$-reduced.

For any $\pi\in\Reps(\GL)$ (not necessarily irreducible) define
\[
\lnrset(\pi)=\{\rho\in\CuspGL:\JacGLnm_{\MS{\rho};*}(\pi)\ne\one\otimes\pi\}=
\{\rho\in\CuspGL:\JacGLnm_{\{\rho\};*}(\pi)\ne0\}.
\]
This is a finite subset of $\CuspGL$ which is nonempty if $\pi\ne0$ unless $\deg\pi=0$.
It follows from the geometric lemma that
\begin{equation} \label{eq: redprod}
\lnrset(\pi_1\times\pi_2)=\lnrset(\pi_1)\cup\lnrset(\pi_2)
\end{equation}
and by Frobenius reciprocity
\begin{equation} \label{eq: ifembedlnr}
\text{if $\pi\hookrightarrow\pi_1\times\pi_2$ with $\pi_1,\pi_2\in\IrrGL$ then }\lnrset(\pi)\supset\lnrset(\pi_1).
\end{equation}
Similarly, define
\[
\rnrset(\pi)=\{\rho\in\CuspGL:\JacGLnm_{*;\MS{\rho}}(\pi)\ne\pi\otimes\one\}=
\{\rho\in\CuspGL:\JacGLnm_{*;\{\rho\}}(\pi)\ne0\}.
\]
Note that
\[
\lnrset(\pi)^\vee=\rnrset(\pi^\vee).
\]

\begin{lemma} (\cite{MR2306606})
For any $\pi\in\IrrGL$ and $A\subset\CuspGL$ there exist $\pi_A\lft\in\IrrGL$ and $L_A(\pi)\in\IrrGL$ satisfying the following conditions:
\begin{enumerate}
\item $\pi\hookrightarrow\pi_A\lft\times L_A(\pi)$, i.e., $\JacGL(\pi)\twoheadrightarrow\pi_A\lft\otimes L_A(\pi)$.
\item $\supp\pi_A\lft\subset A$.
\item $L_A(\pi)$ is left $A$-reduced.
\end{enumerate}
Moreover, $\pi_A\lft$ and $L_A(\pi)$ are uniquely determined by $\pi$ and we have
\[
\comult_{\MS{A};*}(\pi)=\pi_A\lft\otimes L_A(\pi)+\sum_i\alpha_i\otimes\beta_i
\]
where $\deg\beta_i>\deg L_A(\pi)$ for all $i$.
\end{lemma}

\begin{proof}
We recall the simple argument since we will use it repeatedly.
For the existence part, we take $\pi_A\lft$ supported in $A$ of maximal degree with respect to the property that
$\pi\hookrightarrow\pi_A\lft\times\pi'$ for some $\pi'$.
The last part (and the uniqueness) follow from the fact that $\JacGL(\pi)\le\JacGL(\pi_A\lft\times L_A(\pi))$ and that by the geometric lemma we have
\[
\comult(\pi_A\lft\times L_A(\pi))=\pi_A\lft\otimes L_A(\pi)+\sum_i\alpha_i\otimes\beta_i
\]
where $\deg\beta_i>\deg L_A(\pi)$ for all $i$ such that $\supp\alpha_i\subset A$.
\end{proof}

We call $L_A(\pi)$ the \emph{left partial derivative} of $\pi$ with respect to $A$. Analogously, there exist unique $\pi_A\rgt,R_A(\pi)\in\IrrGL$ such that
$\pi\hookrightarrow R_A(\pi)\times\pi_A\rgt$, $\supp\pi_A\rgt\subset A$ and $R_A(\pi)$ is right $A$-reduced.
We have
\[
\comult_{*;\MS{A}}(\pi)=R_A(\pi)\otimes\pi_A\rgt+\sum_i\alpha_i\otimes\beta_i,\ \ \alpha_i,\beta_i\in\IrrGL
\]
where $\deg\alpha_i>\deg R_A(\pi)$ for all $i$.

\begin{lemma} \label{lem: ABder}
For any $\pi\in\IrrGL$ and $A,B\subset\CuspGL$ there exist $\pi_1,\pi_2,\pi_3\in\IrrGL$ such that
\begin{enumerate}
\item $\pi\hookrightarrow\pi_1\times\pi_2\times\pi_3$, i.e., $\JacGLs(\pi)\twoheadrightarrow\pi_1\otimes\pi_2\otimes\pi_3$.
\item $\supp \pi_1\subset A$; $\supp \pi_3\subset B$.
\item $\pi_2$ is left $A$-reduced and right $B$-reduced.
\end{enumerate}
Moreover, $\pi_2$ is uniquely determined and is characterized by the following conditions:
\begin{enumerate}
\renewcommand{\theenumi}{\Alph{enumi}}
\item There exist $\alpha,\gamma\in\IrrGL$ such that $\alpha\otimes\pi_2\otimes\gamma\le\comult_{\MS{A};*;\MS{B}}(\pi)$.
\item If $\alpha\otimes\beta\otimes\gamma\le\comult_{\MS{A};*;\MS{B}}(\pi)$ with $\alpha,\beta,\gamma\in\IrrGL$
then either $\beta=\pi_2$ or $\deg\beta>\deg\pi_2$.
\end{enumerate}
\end{lemma}

\begin{proof}
For the existence we can take $\pi_1=\pi_A\lft$, $\pi_2=R_B(L_A(\pi))$, $\pi_3=(L_A(\pi))_B\rgt$.
The last statement follows once again from the geometric lemma and the fact that $\JacGLs(\pi)\le\JacGLs(\pi_1\times\pi_2\times\pi_3)$.
\end{proof}

We will write $D_{A;B}(\pi)=\pi_2$. It is clear that
\begin{equation} \label{eq: DAB}
D_{A;B}(\pi)=R_B(L_A(\pi))=L_A(R_B(\pi))
\end{equation}
and
\begin{equation} \label{eq: suppDAB}
\supp L_A(\pi)\supset\supp\pi\setminus A, \ \supp R_B(\pi)\supset\supp\pi\setminus B,\ \supp D_{A;B}(\pi)\supset\supp\pi\setminus (A\cup B).
\end{equation}
We also remark that
\[
D_{A;B}(\pi)^\vee=D_{B^\vee;A^\vee}(\pi^\vee).
\]
Note that $\pi_1$ and $\pi_3$ (or even their degrees) are not necessarily uniquely determined by $\pi$ unless $A$ and $B$ are disjoint in which
case $\pi_1=\pi_A\lft$ and $\pi_3=\pi_B\rgt$.

\begin{remark} \label{rem: tensorprodGL}
Let $A$ be a subset of $\CuspGL$.
Denote by $\RepsGL_{\red{A}}$ the Serre ring subcategory of $\RepsGL$ consisting of left $A$-reduced representations.
Let $(\IrrGL)_{\red{A}}=\Irr\RepsGL_{\red{A}}\subset\IrrGL$ and $\GrGL_{\red{A}}=\Gr(\RepsGL_{\red{A}})\subset\GrGL$.
Assume that $\rshft{A}=A$. Then there is no difference between left $A$-reduced and right $A$-reduced.
Denote the complement of $A$ by $A^c$. The map $\times$ defines a bijection
\[
(\IrrGL)_{\red{A}}\times (\IrrGL)_{\red{A^c}}\rightarrow\IrrGL
\]
which gives rise to an isomorphism of rings
\[
\GrGL_{\red{A}}\otimes\GrGL_{\red{A^c}}\simeq\GrGL.
\]
In fact $\times$ induces an equivalence of ring categories
\[
\RepsGL_{\red{A}}\otimes\RepsGL_{\red{A^c}}\simeq\RepsGL.
\]
\end{remark}

\subsection{Zelevinsky classification} (\cite{MR584084})
A \emph{segment} is a non-empty finite subset $\Delta$ of $\CuspGL$ of the form $\Delta=\{\rho_1,\dots,\rho_k\}$ where $\rho_{i+1}=\rshft{\rho}_i$, $i=1,\dots,k-1$.
Denote the set of segments by $\Seg$. Given $\Delta=\{\rho_1,\dots,\rho_k\}\in\Seg$ as before,
the representations $\rho_1\times\dots\times\rho_k$ and $\rho_k\times\dots\times\rho_1$ are \SI. We write
\begin{gather*}
b(\Delta)=\rho_1, e(\Delta)=\rho_k,\\
Z(\Delta)=\soc(\rho_1\times\dots\times\rho_k),\ \ L(\Delta)=\soc(\rho_k\times\dots\times\rho_1),\\
\deg\Delta=\deg Z(\Delta)=\deg L(\Delta)=\deg\rho_1+\dots+\deg\rho_k=k\cdot\deg\rho_1,\\
\data_\Delta=\data_{Z(\Delta)}=\data_{L(\Delta)}=\rho_1+\dots+\rho_k\in\MS{\CuspGL},\\
\expo(\Delta)=\expo(Z(\Delta))=\expo(L(\Delta))=(\expo(\rho_1)+\dots+\expo(\rho_k))/k=(\expo(\rho_1)+\expo(\rho_k))/2,\\
\lextend\Delta=\{\rho_0,\dots,\rho_k\}\in\Seg, \Delta^+=\{\rho_1,\dots,\rho_{k+1}\}\in\Seg, \text{ where }\rho_0=\lshft{\rho_1}, \rho_{k+1}=\rshft{\rho_k},\\
\lshrink\Delta=\{\rho_2,\dots,\rho_k\}\in\Seg\cup\{\emptyset\}, \Delta^-=\{\rho_1,\dots,\rho_{k-1}\}\in\Seg\cup\{\emptyset\},\\
\lshft{\Delta}=\{\lshft{\rho}_1,\dots,\lshft{\rho}_k\}\in\Seg,
\rshft{\Delta}=\{\rshft{\rho}_1,\dots,\rshft{\rho}_k\}\in\Seg, \Delta^\vee=\{\rho_k^\vee,\dots,\rho_1^\vee\}\in\Seg,\\
\text{so that $Z(\Delta^\vee)=Z(\Delta)^\vee$ and similarly for $\rshft{}$ and $\lshft{}$}.
\end{gather*}
For compatibility we also write $Z(\emptyset)=L(\emptyset)=\one$.
If $\rho\in\CuspGL$ and $n\ge-1$ we write
\[
[\rho,\rho[n]]=\{\rho,\rho[1],\dots,\rho[n]\}\in\Seg\cup\{\emptyset\}.
\]

We have
\begin{equation} \label{eq: Jacseg}
\JacGL(Z(\Delta))=\oplus_{\rho\in\lextend\Delta}Z([b(\Delta),\rho])\otimes Z([\rshft{\rho},e(\Delta)]).
\end{equation}

If $\Delta_1, \Delta_2\in\Seg$ we write $\Delta_1\prec\Delta_2$ if $b(\Delta_1)\notin\Delta_2$, $b(\lshft{\Delta}_2)\in\Delta_1$ and
$e(\Delta_2)\notin\Delta_1$. In this case $\soc(Z(\Delta_1)\times Z(\Delta_2))=Z(\Delta_1')\times Z(\Delta'_2)$ where
$\Delta_1'=\Delta_1\cup\Delta_2$, $\Delta_2'=\Delta_1\cap\Delta_2$ (the latter is possibly empty). Note that
\begin{equation} \label{eq: expoge}
\expo(\Delta_1')>\expo(\Delta_1)\text{ and if $\Delta'_2\ne\emptyset$ then }\expo(\Delta_2')>\expo(\Delta_1).
\end{equation}
If either $\Delta_1\prec\Delta_2$ or $\Delta_2\prec\Delta_1$ (we cannot have both) then we say that $\Delta_1$ and $\Delta_2$ are linked.
The representation $Z(\Delta_1)\times Z(\Delta_2)$ is reducible if and only if $\Delta_1$ and $\Delta_2$ are linked,
in which case it is of length two.

A \emph{multisegment} is an element $\m$ of $\MS{\Seg}$. Thus, $\m=\Delta_1+\dots+\Delta_k$ for some $\Delta_1,\dots,\Delta_k\in\Seg$.
We write
\begin{gather*}
\deg\m=\deg\Delta_1+\dots+\deg\Delta_k,\ \supp\m=\Delta_1\cup\dots\cup\Delta_k\subset\CuspGL,\\
\m^\vee=\Delta_1^\vee+\dots+\Delta_k^\vee\in\MS{\Seg},\ \data_\m=\data_{\Delta_1}+\dots+\data_{\Delta_k}\in\MS{\CuspGL}.
\end{gather*}
We may enumerate the $\Delta_i$'s so that $\Delta_i\not\prec\Delta_j$ whenever $i<j$.
In this case, the representation $\std{\m}=Z(\Delta_1)\times\dots\times Z(\Delta_k)$ (resp., $\cstd{\m}=L(\Delta_1)\times\dots\times L(\Delta_k)$)
is \SI\ (resp., \CSI) and depends only on $\m$. The maps
\[
\m\mapsto Z(\m):=\soc(\std{\m}),\ \ \m\mapsto L(\m):=\cos(\cstd{\m})
\]
are bijections between $\MS{\Seg}$ and $\Irr$. For any $\m\in\MS{\Seg}$ we have
\begin{gather*}
\deg Z(\m)=\deg L(\m)=\deg\m,\ \data_{Z(\m)}=\data_{L(\m)}=\data_\m,\\\supp Z(\m)=\supp L(\m)=\supp\m,\ Z(\m)^\vee=Z(\m^\vee), L(\m)^\vee=L(\m^\vee),
L(\m)=Z(\m)^t.
\end{gather*}
Moreover, if $\m_1,\m_2\in\MS{\Seg}$ then $Z(\m_1+\m_2)$ occurs with multiplicity one in $\JH(Z(\m_1)\times Z(\m_2))$; similarly for $L(\m_1+\m_2)$.
We have

For simplicity, if $\m$ is a multisegment we write $\lnrset(\m)=\lnrset(Z(\m))$ and $\rnrset(\m)=\rnrset(Z(\m))$.
We have the following combinatorial description of $\lnrset(\m)$ and $\rnrset(\m)$.
\begin{proposition} (\cite[Theorem 5.11]{MR3573961} which is based on \cite{MR2306606} and \cite{MR2527415}) \label{prop: combprop}
Let $\m=\Delta_1+\dots+\Delta_k$ be a multisegment. For any $\rho\in\CuspGL$ let
\[
X_\rho=\{i:b(\Delta_i)=\rho\}\text{ and }Y_\rho=\{i:e(\Delta_i)=\rho\}.
\]
Then $\rho\notin\lnrset(\m)$ if and only if there exists an injective map $f:X_\rho\rightarrow X_{\rshft\rho}$ such that
$\Delta_i\prec\Delta_{f(i)}$ for all $i\in X_\rho$. Moreover, there exists a subset $A\subset X_\rho$ such that
\begin{equation} \label{eq: Lrho}
L_\rho(Z(\m))=Z(\m+\sum_{i\in A}(\lshrink\Delta_i-\Delta_i)).
\end{equation}
Similarly, $\rho\notin\rnrset(\m)$ if and only if there exists an injective map $f:Y_\rho\rightarrow Y_{\lshft\rho}$ such that
$\Delta_{f(i)}\prec\Delta_i$ for all $i\in Y_\rho$; there exists a subset $A\subset Y_\rho$ such that
\begin{equation} \label{eq: Rrho}
R_\rho(Z(\m))=Z(\m+\sum_{i\in A}(\Delta_i^--\Delta_i)).
\end{equation}
\end{proposition}

The following consequence will be useful.
\begin{corollary} \label{cor: rhoinsuppD}
Suppose that $\Delta\le\m$ and $\rho\in\lshrink\Delta$ (i.e., $\rho\in\Delta$ but $\rho\ne b(\Delta)$).
Then $\rho\in\supp L_\rho(Z(\m))$.
Hence $\rho\in\supp D_{\rho;\rho'}(Z(\m))$ if $\rho'\ne e(\Delta)$ or $\rho'\ne\rho$.
\end{corollary}

Indeed, the first statement follows from \eqref{eq: Lrho}. The second statement follows from \eqref{eq: DAB} and \eqref{eq: Rrho}.

\subsection{Ladder representations}
We define a partial order on $\Cusp$ by $\rho_1\le\rho_2$ if $\rho_2=\rho_1[n]$ for some $n\in\Z_{\ge0}$.

Recall that a \emph{ladder} (cf.~\cite{MR3163355}) is a multisegment of the form
\begin{equation} \label{eq: decep}
\m=\Delta_1+\dots+\Delta_k\text{ with }b(\Delta_k)<\dots<b(\Delta_1)\text{ and }e(\Delta_k)<\dots<e(\Delta_1).
\end{equation}
By \cite{MR2996769}, if $\m$ is a ladder then
\begin{equation} \label{eq: comladder}
\JacGL(Z(\m))=\sum_{\substack{\rho_i\in\lextend\Delta_i,\\\rho_1>\dots>\rho_k}}
Z(\sum_{i=1}^k[b(\Delta_i),\rho_i])\otimes Z(\sum_{i=1}^k[\rshft{\rho_i},e(\Delta_i)])
\end{equation}
where each summand is the tensor product of two ladders.
In particular,
\begin{equation} \label{eq: ladlftred}
\lnrset(\m)=\{b(\Delta_i):1\le i\le k\}\setminus \{b(\lshft{\Delta_i}):1\le i\le k\}.
\end{equation}
We will need another fact
\begin{lemma} \label{lem: length2}
Suppose that $\m$ is a ladder as in \eqref{eq: decep} and let $\Delta_{k+1}\in\Seg$ with $\Delta_{k+1}\prec\Delta_k$.
Then we have a short exact sequence
\[
0\rightarrow Z(\n)\rightarrow Z(\Delta_{k+1})\times Z(\m)\rightarrow Z(\Delta_1+\dots+\Delta_{k+1})\rightarrow0
\]
where
\[
\n=\Delta_1+\dots+\Delta_{k-1}+\Delta_k\cup\Delta_{k+1}+\Delta_k\cap\Delta_{k+1}.
\]
Moreover,
\[
Z(\Delta_{k+1}\setminus\Delta_k)\otimes\tau\le \comult(Z(\n))\text{ for some }\tau\in\IrrGL.
\]
\end{lemma}

\begin{proof}
Let $\Pi=Z(\Delta_{k+1})\times Z(\m)$ and $\m'=\Delta_1+\dots+\Delta_{k+1}$. Note that
\[
\Pi^\vee\simeq Z(\Delta_{k+1}^\vee)\times Z(\m^\vee)\hookrightarrow Z(\Delta_{k+1}^\vee)\times\std{\m^\vee}=\std{\m'^\vee}.
\]
Thus, $\soc(\Pi^\vee)=Z(\m'^\vee)$, or equivalently
\[
\Pi\twoheadrightarrow Z(\m').
\]
On the other hand, let $\tau=Z(\Delta_k\cap\Delta_{k+1})\times Z(\m)$ which is irreducible.
(For instance, this easily follows from \cite[Corollary 5.14]{MR3573961}.)
Then
\[
\Pi\hookrightarrow Z(\Delta_{k+1}\setminus\Delta_k)\times\tau
\]
and by the recipe of \cite[Proposition 5.6 and Theorem 5.11]{MR3573961}
(which easily reduces to the case where $\Delta_{k+1}\setminus\Delta_k$ is a singleton) we have
\[
\soc(\Pi)=\soc(Z(\Delta_{k+1}\setminus\Delta_k)\times\tau)=Z(\n).
\]
We conclude that
\[
\JacGL(Z(\n))\twoheadrightarrow Z(\Delta_{k+1}\setminus\Delta_k)\otimes\tau.
\]
To finish the proof of the lemma, it remains to show that $\Pi$ is of length two.
We will prove it by induction on $\deg\m$. The base of the induction is trivial.

For the induction step, note that by \eqref{eq: redprod} and \eqref{eq: ladlftred} we have
\[
\lnrset(\Pi)=b(\Delta_{k+1})\cup\lnrset(\m)=b(\Delta_{k+1})\cup\lnrset(\m')
\]
and
\[
\lnrset(\m')=\lnrset(\m)\cup
\begin{cases}b(\Delta_{k+1})&\text{if }b(\Delta_{k+1})\ne b(\lshft{\Delta_k}),\\\emptyset&\text{otherwise.}\end{cases}
\]
Moreover, by \eqref{eq: comladder} and the geometric lemma, for any $\rho\in\lnrset(\Pi)$ we have
\[
\JacGLnm_{\{\rho\};*}(\Pi)=\rho\otimes\begin{cases}Z(\lshrink\Delta_{k+1})\times Z(\m)&\text{if }\rho=b(\Delta_{k+1}),\\
Z(\Delta_{k+1})\times Z(\m^{(i)})&\text{if }\rho=b(\Delta_i),i=1,\dots,k,\end{cases}
\]
where $\m^{(i)}=\Delta_1'+\dots+\Delta'_k$ with $\Delta_i'=\lshrink\Delta_i$ and $\Delta'_j=\Delta_j$ if $j\ne i$.
Note that $\m^{(i)}$ is a ladder since $b(\rshft{\Delta_i})\ne b(\Delta_{i-1})$ if $b(\Delta_i)\in\lnrset(\m)$.
Thus, by induction hypothesis we have for any $\rho\in\lnrset(\Pi)$
\[
\ell_\rho(\Pi):=\ell(\JacGLnm_{\{\rho\};*}(\Pi))=\begin{cases}1&\text{if }\rho=b(\Delta_{k+1})=b(\lshft{\Delta}_k)\text{ or }\rho=b(\Delta_k)=e(\rshft{\Delta}_{k+1}),\\
2&\text{otherwise.}\end{cases}
\]
On the other hand, it follows from \eqref{eq: ifembedlnr}, \eqref{eq: ladlftred} and the fact that
\[
Z(\n)\hookrightarrow Z(\Delta_1+\dots+\Delta_{k-1}+\Delta_k\cup\Delta_{k+1})\times Z(\Delta_k\cap\Delta_{k+1})
\]
and
\[
Z(\n)\hookrightarrow Z(\Delta_1+\dots+\Delta_{k-1}+\Delta_k\cap\Delta_{k+1})\times Z(\Delta_k\cup\Delta_{k+1})
\]
that
\[
\lnrset(\n)=\lnrset(\Pi)\setminus\begin{cases}b(\Delta_k)&\text{if }b(\Delta_k)=e(\rshft{\Delta}_{k+1}),\\
\emptyset&\text{otherwise.}\end{cases}
\]
It follows that
\[
\sum_{\rho\in\lnrset(\Pi)}\ell_\rho(\Pi)=\#\lnrset(\m')+\#\lnrset(\n)
\]
and in particular,
\[
\sum_{\rho\in\lnrset(\Pi)}\ell_\rho(\Pi)\le \sum_{\rho\in\lnrset(\m')}\ell_\rho(Z(\m'))+\sum_{\rho\in\lnrset(\n)}\ell_\rho(Z(\n)).
\]
This implies that $\Pi\le Z(\m')+Z(\n)$ since for any $0\ne\pi'\le\Pi$ we have $\JacGLnm_{\{\rho\};*}(\pi')\ne0$
for some $\rho\in\lnrset(\Pi)$. Thus $\Pi$ is of length two.
\end{proof}

Passing to the contragredient, we get
\begin{corollary} \label{cor: length2}
Suppose that $\m=\Delta_1+\dots+\Delta_k$ is a ladder and $\Delta_1\prec\Delta_0$. Then we have a short exact sequence
\[
0\rightarrow Z(\n)\rightarrow Z(\m)\times Z(\Delta_0)\rightarrow Z(\Delta_0+\dots+\Delta_k)\rightarrow0
\]
where
\[
\n=\Delta_0\cup\Delta_1+\Delta_0\cap\Delta_1+\Delta_2+\dots+\Delta_k.
\]
Moreover,
\[
\tau\otimes Z(\Delta_0\setminus\Delta_1)\le \comult(Z(\n))\text{ for some }\tau\in\IrrGL.
\]
\end{corollary}

\section{Classical groups} \label{sec: classical}
Next, we turn to classical groups which are the main object of the paper.
Again, most of the results in this section are standard.
\subsection{}

Let $E$ be either $F$ or a quadratic Galois extension of $F$.
In the former case let $\Galinv$ be the Galois involution of $E/F$.
In the latter case $\Galinv=\id$.

\emph{All the notation of the previous section will be used with respect to $E$.}
Since we work with groups over $F$ this means that formally $\GL_n$ should be replaced by its restriction of scalars with respect to $E/F$.
In order to avoid extra notation we use this convention implicitly throughout.

Let $\iota$ be the involution $g\mapsto\tensor[^t]g{^{-1}}$ of the general linear group and let $\tilde{}$ be the composition
of $\iota$ with $\Galinv$ (which commute). We use the same notation for the induced actions on $\RepsGL$, $\GrGL$, $\Seg$, $\MS{\Seg}$ etc..
(For convenience we sometimes write $\tilde Z(\m)$ for $\widetilde{Z(\m)}=Z(\tilde\m)$.)
Note that $\iota$ is a covariant functor of $\RepsGL$. We have
\begin{subequations}\label{eq: iotaind}
\begin{gather}
\iota(\pi)\simeq\pi^\vee\text{ if $\pi\in\IrrGL$ and hence $\iota$ coincides with the contragredient on $\GrGL$.}\\
\iota(\pi_1\times\pi_2)\simeq\iota(\pi_2)\times\iota(\pi_1)\ \ \ \pi_1,\pi_2\in\RepsGL.
\end{gather}
\end{subequations}
Also, $\expo(\pi^\Galinv)=\expo(\pi)$ and $\expo(\tilde\pi)=-\expo(\pi)$ for any $\pi\in\RepsGL$ which admits a central character.

We consider an anisotropic $\epsilon$-hermitian space $V_0$ over $E$ with $\epsilon\in\{\pm1\}$.
Thus, $V_0$ is trivial in the symplectic case, of dimension $\le4$ in the quadratic case and of dimension $\le2$ in the hermitian case.
We then have a tower $V_n$, $n\ge0$ of $\epsilon$-hermitian spaces where $V_n$ is obtained from $V_0$ by adding
$n$ copies of a hyperbolic plane. Consider the sequence of isometry groups $G_n=\Isom(V_n)$ (of $F$-rank $n$).
(See \cite[Chapitre 1]{MR1041060} for basic facts about classical groups.)
Note that the center of $G_n$ is anisotropic.
Let
\[
\RepsG=\oplus_{n\ge0}\Reps(G_n),\ \ \IrrG=\Irr\RepsG=\dunion_{n\ge0}\Irr G_n,\ \ \CuspG=\dunion_{n\ge0}\Cusp G_n.
\]
(Note that $\Irr G_0\subset\CuspG$ and in particular $\one\in\CuspG$ if $G_0=1$.
On the other hand, if $V_0$ is the $0$-dimensional quadratic space then $\Cusp G_1=\emptyset$.)
As before we write $\deg\pi=n$ if $\pi\in\Irr G_n$.
For any $m\le n$ the stabilizer of a totally isotropic $m$-dimensional subspace $U$ of $V_n$ is a parabolic subgroup of $G_n$, which is defined over $F$
and up to conjugation uniquely determined by $m$. The Levi part is canonically $\GL(U)\times\Isom(U^\perp/U)$ which we identify with $\GL_m\times G_{n-m}$.
This gives rise to parabolic induction
\[
\sdp:\Reps(\GL_m)\times\Reps(G_n)\rightarrow\Reps(G_{n+m}),\ \ n,m\ge0, \ \ \ \sdp:\RepsGL\times\RepsG\rightarrow\RepsG
\]
which are bilinear biexact bifunctors with an associativity constraint.
Thus, $\RepsG$ is a left module category over $\RepsGL$ in the sense of \cite[\S 7.1]{MR3242743}.
The Grothendieck group $\GrG$ of $\RepsG$ becomes a left $\GrGL$-module.
We will continue to denote the action of $\GrGL$ on $\GrG$ by $\sdp$.
This gives rise to an action of $\GrGL\otimes\GrGL$ on $\GrGL\otimes\GrG$ (also denoted by $\sdp$) given by
\[
(\alpha\otimes\beta)\sdp(\gamma\otimes\delta)=\alpha\times\gamma\otimes\beta\sdp\delta.
\]
We have
\begin{equation} \label{eq: indtildeinv}
\JH(\pi\sdp\sigma)=\JH(\tilde\pi\sdp\sigma),\ \ \pi\in\RepsGL,\ \sigma\in\RepsG.
\end{equation}

Once again, the left-adjoint
\[
\JacG:\RepsG\rightarrow\RepsGL\otimes\RepsG=\oplus_{n,m\ge0}\Reps(\GL_n\times G_m)
\]
of $\sdp$ is given by $\oplus_{n\ge0}(\oplus_{n_1+n_2=n}\JacG_{n_1;n_2})$ where
\[
\JacG_{n;m}:\Reps(G_{n+m})\rightarrow\Reps(\GL_n\times G_m),\ \ n,m\ge0
\]
is the normalized Jacquet functor.
On the level of Grothendieck groups we get a map
\[
\comultG:\GrG\rightarrow\GrGL\otimes\GrG.
\]
(By abuse of notation we sometimes consider $\comultG$ as a map from $\RepsG$ via $\JH$.)
Using the geometric lemma (see \cite{MR1356358, MR1739616} and the comments in \S1 and \S15 of \cite{MR1896238}) $\comultG$ satisfies
\begin{equation} \label{eq: Tadic formula}
\comultG(\alpha\sdp\beta)=\comod(\alpha)\sdp\comultG(\beta),\ \ \alpha\in\GrGL,\ \beta\in\GrG
\end{equation}
where
\[
\comod:\GrGL\rightarrow\GrGL\otimes\GrGL
\]
is the ring homomorphism corresponding to the composition of exact functors
\begin{equation} \label{eq: functorcomod}
\RepsGL\xrightarrow{\JacGLs}\RepsGL\otimes\RepsGL\otimes\RepsGL\xrightarrow{\id\otimes s}\RepsGL\otimes\RepsGL\otimes\RepsGL
\xrightarrow{\times\otimes\id}\RepsGL\otimes\RepsGL
\end{equation}
where $s(\alpha\otimes\beta)=\tilde\beta\otimes\alpha$. Thus,
\begin{equation} \label{eq: comod}
\text{if $\comults(\pi)=\sum_i\alpha_i\otimes\beta_i\otimes\gamma_i$ then }\comod(\pi)=\sum_i\alpha_i\times\tilde{\gamma_i}\otimes\beta_i.
\end{equation}
Note that
\[
\comod(\pi)=\comod(\tilde\pi)
\]
in accordance with \eqref{eq: indtildeinv} and \eqref{eq: Tadic formula}.

We will also let
\[
\comodmax:\GrGL\rightarrow\GrGL
\]
be the homomorphism corresponding to
\[
\RepsGL\xrightarrow{\JacGL}\RepsGL\otimes\RepsGL\xrightarrow{s}\RepsGL\otimes\RepsGL\xrightarrow{\times}\RepsGL,
\]
i.e., to the composition of \eqref{eq: functorcomod} with the functor
\[
\RepsGL\otimes\RepsGL\xrightarrow{\id\otimes p_0}\RepsGL\otimes\Reps(\GL_0)=\RepsGL
\]
where $p_0$ is the projection of $\RepsGL=\oplus_{n\ge0}\Reps(\GL_n)$ to $\Reps(\GL_0)$ (the category of finite-dimensional vector spaces).
Thus, $\comodmax=\times\circ s\circ\comult$.
Explicitly,
\begin{equation} \label{eq: comodmax}
\text{if $\comult(\pi)=\sum_i\alpha_i\otimes\beta_i$ then }\comodmax(\pi)=\sum_i\alpha_i\times\tilde{\beta_i}.
\end{equation}

For example, for any $\Delta\in\Seg$ we have (using \eqref{eq: Jacseg})
\begin{equation} \label{eq: comodseg}
\comod(Z(\Delta))=\sum_{\rho,\rho'\in\lextend\Delta:\rho\le\rho'}Z([b(\Delta),\rho])\times\tilde Z([\rshft{\rho'},e(\Delta)])\otimes Z([\rshft{\rho},\rho'])
\end{equation}
and
\begin{equation} \label{eq: comodmaxseg}
\comodmax(Z(\Delta))=\sum_{\rho\in\lextend\Delta}Z([b(\Delta),\rho])\times\tilde Z([\rshft{\rho},e(\Delta)]).
\end{equation}
More generally, for any ladder $\m$ as in \eqref{eq: decep} we have
\begin{equation} \label{eq: comodladder}
\comod(Z(\m))=
\sum_{\substack{\rho_i,\rho'_i\in\lextend\Delta_i\text{ for all }i\\\rho_1>\dots>\rho_k,\rho'_1>\dots>\rho'_k\\
\rho'_i\ge\rho_i\text{ for all }i}}
Z(\sum_{i=1}^k[b(\Delta_i),\rho_i])\times\tilde Z(\sum_{i=1}^k[\rshft{\rho_i'},e(\Delta_i)])\otimes Z(\sum_{i=1}^k[\rshft{\rho_i},\rho'_i])
\end{equation}
and
\begin{equation} \label{eq: comodmaxladder}
\comodmax(Z(\m))=\sum_{\substack{\rho_i\in\lextend\Delta_i\text{ for all }i\\\rho_1>\dots>\rho_k}}
Z(\sum_{i=1}^k[b(\Delta_i),\rho_i])\times\tilde Z(\sum_{i=1}^k[\rshft{\rho_i},e(\Delta_i)]).
\end{equation}

\subsection{}
The following is an immediate consequence of Frobenius reciprocity (cf.~\cite[Lemma 2.5]{MR3573961}).
\begin{lemma} \label{lem: simpleSIcrit}
Suppose that $\pi\in\RepsGL$ and $\sigma\in\RepsG$ are \SI\ and that $\soc(\pi)\otimes\soc(\sigma)$ occurs with multiplicity one in $\comultG(\pi\sdp\sigma)$.
Then $\pi\sdp\sigma$ is also \SI. In particular, (by \eqref{eq: obvious}) $\soc(\pi\sdp\sigma)=\soc(\soc(\pi)\sdp\soc(\sigma))$.
\end{lemma}

Let $X\subset\MS{\CuspGL}$.
As before, we write $\JacG_{X;*}$ for the composition of $\JacG$ with
\[
\RepsGL\otimes\RepsG\xrightarrow{\prcspdata_X\otimes\id}\RepsGL_X\otimes\RepsG.
\]
We will also write
\[
\comultG_{X;*}:\GrG\rightarrow\GrGL_X\otimes\GrG
\]
for the corresponding map of Grothendieck groups. Similarly, let
\[
\comod_{X;*}:\GrGL\rightarrow\GrGL_X\otimes\GrGL
\]
be the composition of $\comod$ with $\prcspdata_X\otimes\id$. For any $A\subset\CuspGL$ we have
\begin{equation} \label{eq: comodA}
\comod_{\MS{A};*}=(\times\otimes\id)\circ(\id\otimes s)\otimes\comult_{\MS{A};*;\MS{\tilde A}}
\end{equation}
and
\begin{equation} \label{eq: Tadic formulaA}
\comultG_{\MS{A};*}(\alpha\sdp\beta)=\comod_{\MS{A};*}(\alpha)\sdp\comultG_{\MS{A};*}(\beta),\ \ \alpha\in\GrGL,\ \beta\in\GrG.
\end{equation}

\begin{definition}
Let $\rho\in\CuspGL$.
We say that $\sigma\in\RepsG$ is $\rho$-reduced if $\JacG_{\{\rho\};*}(\sigma)=0$.
For a subset $A\subset\CuspGL$ we say that $\sigma\in\RepsG$ is $A$-reduced
if it is $\rho$-reduced for all $\rho\in A$.
\end{definition}

Note that if $\sigma\in\IrrG$ then $\sigma$ is $\rho$-reduced if and only if there does not exist $\sigma'\in\IrrG$ such that $\sigma\hookrightarrow\rho\sdp\sigma'$.

The following is proved using the same argument as in Lemma \ref{lem: ABder}.
\begin{lemma} \label{lem: derivclass}
For any $\sigma\in\IrrG$ and $A\subset\CuspGL$ there exist $\pi\in\IrrGL$ and $\sigma'\in\IrrG$ such that:
\begin{enumerate}
\item $\sigma\hookrightarrow\pi\sdp\sigma'$.
\item $\supp\pi\subset A$.
\item $\sigma'$ is $A$-reduced.
\end{enumerate}
Moreover, $\sigma'$ is uniquely determined by $\sigma$ and is characterized by the following properties:
\begin{enumerate}
\renewcommand{\theenumi}{\Alph{enumi}}
\item There exists $\alpha\in\IrrGL$ such that $\alpha\otimes\sigma'\le\comultG_{\MS{A};*}(\sigma)$.
\item If $\alpha\otimes\beta\le\comultG_{\MS{A};*}(\sigma)$ with $\alpha\in\IrrGL$ and $\beta\in\IrrG$ then either $\beta=\sigma'$
or $\deg\beta>\deg\sigma'$.
\end{enumerate}
\end{lemma}

We write $D_A(\sigma)=\sigma'$.

We will need the following result which is a direct consequence of \cite[Proposition 1.3]{MR2504024}.
\begin{lemma} \label{lem: noisotypic}
For any $\pi\in\IrrGL$ and $\sigma\in\IrrG$ there exists $\sigma'\in\IrrG$ which occurs with multiplicity one in $\JH(\pi\sdp\sigma)$.
\end{lemma}

\begin{corollary} \label{cor: redAAvee}
Suppose that $\pi\sdp\sigma$ is irreducible.
Then $D_{A;\tilde A}(\pi)\sdp D_A(\sigma)=D_A(\pi\sdp\sigma)$.
In particular, $D_{A;\tilde A}(\pi)\sdp D_A(\sigma)$ is irreducible.
\end{corollary}

\begin{proof}
By Lemma \ref{lem: noisotypic} it is enough to show that
\begin{equation} \label{eq: isot}
D_{A;\tilde A}(\pi)\sdp D_A(\sigma)\text{ is a multiple of }D_A(\pi\sdp\sigma)\text{ in the Grothendieck group}.
\end{equation}
Indeed, it follows from \eqref{eq: comodA}, \eqref{eq: Tadic formulaA} and Lemma \ref{lem: ABder} that
\[
\comultG_{\MS{A};*}(\pi\sdp\sigma)\ge\pi'\otimes D_{A;\tilde A}(\pi)\sdp D_A(\sigma)\text{ for some }\pi'\in\IrrGL
\]
while if $\comultG_{\MS{A};*}(\pi\sdp\sigma)\ge\pi'\otimes\sigma'$ for some $\pi'\in\IrrGL$ and $\sigma'\in\IrrG$ then
$\deg\sigma'\ge\deg D_{A;\tilde A}(\pi)+\deg D_A(\sigma)$.
Hence, \eqref{eq: isot} follows from Lemma \ref{lem: derivclass}.
\end{proof}

\begin{remark} (Compare with Remark \ref{rem: tensorprodGL}.)
Let $A$ be a subset of $\CuspGL$.
Denote by $\RepsG_{\red{A}}$ the Serre subcategory of $\RepsG$ consisting of $A$-reduced representations,
$(\IrrG)_{\red{A}}=\Irr\RepsG_{\red{A}}\subset\IrrG$ and $\GrG_{\red{A}}=\Gr(\RepsG_{\red{A}})\subset\GrG$.
Suppose that $\tilde A=A=\rshft{A}$. Then $\pi\in(\IrrG)_{\red{A}}$ if and only if the cuspidal data of $\pi$ is of the form
$\rho_1+\dots+\rho_k;\sigma$ where $\rho_i\notin A$ for all $i$. (The $\rho_i$'s are defined up to $\tilde{}$.)
The main result of \cite{MR1481814} is that the map
\[
\pi\mapsto (D_A(\pi),D_{A^c}(\pi))
\]
defines a bijection between $\IrrG$ and $\IrrG_{\red{A}}\times\IrrG_{\red{A^c}}$ which induces an isomorphism
\[
\GrG\simeq\GrG_{\red{A}}\otimes\GrG_{\red{A^c}}
\]
of $\GrGL=\GrGL_{\red{A}}\otimes\GrGL_{\red{A^c}}$-modules.
It is likely that there is an equivalence of module categories over $\RepsGL\simeq\RepsGL_{\red{A}}\otimes\RepsGL_{\red{A^c}}$
between $\RepsG$ and $\RepsG_{\red{A}}\otimes\RepsG_{\red{A^c}}$.
See \cite{1502.04357} for a related result.
\end{remark}

\subsection{Classification}
Recall that by Casselman's criterion, a representation $\sigma\in\IrrG$ is tempered if and only if $\expo(\pi)\ge0$ whenever
$\pi\otimes\sigma'\le\comultG(\sigma)$, $\pi\in\IrrGL$, $\sigma'\in\IrrG$.
(Cf. \cite[\S16]{MR1896238} for the even orthogonal case.)
Dually, a representation $\sigma\in\IrrG$ is called \emph{cotempered}\footnote{Note that in \cite{MR2450724} these representations were called negative.
However, we prefer to call them cotempered to emphasize the analogy with tempered representations.}
if $\expo(\pi)\le0$ whenever $\pi\otimes\sigma'\le\comultG(\sigma)$, $\pi\in\IrrGL$, $\sigma'\in\IrrG$.
We denote by $\IrrtmpG$ (resp., $\IrrctmpG$) the set of irreducible tempered (resp., cotempered) representations of $G_n$, $n\ge0$.
Thus, $(\IrrtmpG)^t=\IrrctmpG$.

A segment $\Delta\in\Seg$ is called positive (resp., non-negative) if $\expo(\Delta)>0$ (resp., $\expo(\Delta)\ge0$).
We denote by $\Seg_{>0}$ (resp., $\Seg_{\ge0}$) the set of positive (resp., non-negative) segments.
An element $\m$ of $\MS{\Seg_{>0}}$ (resp., $\MS{\Seg_{\ge0}}$) is called a positive (resp., non-negative) multisegment.
In this case we will also call the representation $Z(\m)$ \Zpstv\ (to emphasize the relation to the Zelevinsky classification).

The Langlands classification for classical groups asserts that for any positive multisegment $\m$ and $\theta\in\IrrtmpG$ the representation
$\cstd{\m}\sdp\theta$ is \CSI\ and the map
\[
(\m,\theta)\mapsto L(\m;\theta):=\cos(\cstd{\m}\sdp\theta)
\]
is a bijection between $\MS{\Seg_{>0}}\times\IrrtmpG$ and $\IrrG$. (See \cite[Appendix]{MR2017065} for the split even orthogonal case.)

Dually, we have the following
\begin{theorem} \cite{MR2450724}\footnote{[loc. cit.] does not treat the unitary case, but it can be handled the same way.
We omit the details.}
\label{thm: Zelclass}
\begin{enumerate}
\item For any positive multisegment $\m$ and $\theta\in\IrrctmpG$ the representation
\[
\std{\m;\theta}:=\std{\m}\sdp\theta
\]
is \SI.
\item The map
\[
(\m,\theta)\mapsto Z(\m;\theta)=\soc\std{\m;\theta}
\]
is a bijection between $\MS{\Seg_{>0}}\times\IrrctmpG$ and $\IrrG$.
\item \label{part: stdirred} $\std{\m}\sdp\theta$ is irreducible if and only if
\begin{enumerate}
\item $Z(\Delta)\sdp\theta$ is irreducible for all $\Delta\le\m$, and,
\item $Z(\Delta)\times Z(\Delta')$ and $Z(\Delta)\times Z(\widetilde{\Delta'})$ are irreducible whenever $\Delta+\Delta'\le\m$.
\end{enumerate}
\item \label{part: Zvee} $Z(\m;\theta)^\vee=Z(\m^{\Galinv};\theta^\vee)$ for any $\m\in\MS{\Seg_{>0}}$ and $\sigma\in\IrrctmpG$.
\item \label{part: invol} $L(\m;\theta)^t=Z(\m;\theta^t)$.
\end{enumerate}
\end{theorem}

Note that part \ref{part: invol} is implicit in \cite{MR2450724} but it can be proved using a well-known argument of Rodier \cite{MR689531}.
Namely, for a fixed supercuspidal data we prove the relation $L(\m;\theta)^t=Z(\m;\theta^t)$ by induction on $\beta(\m)$
where $\beta(\m)$ is defined as follows. Suppose that $\m=\Delta_1+\dots+\Delta_k$ with $\expo(\Delta_1)\ge\dots\ge\expo(\Delta_k)>0$
and $G$ has rank $n$. Let $(x_1,\dots,x_n)\in\R^n$ be the vector
\[
(\overbrace{\expo(\Delta_1),\dots,\expo(\Delta_1)}^{\deg\Delta_1},\overbrace{\expo(\Delta_2),\dots,\expo(\Delta_2)}^{\deg\Delta_2},\dots,
\overbrace{\expo(\Delta_k),\dots,\expo(\Delta_k)}^{\deg\Delta_k},0,\dots,0).
\]
Then $\beta(\m):=\sum_{i=1}^n (n+1-i)x_i$.
We have (see e.g. \cite{MR2504024})
\[
\cstd{\m}\sdp\theta=L(\m;\theta)+\sum_iL(\m_i;\theta_i)
\]
with $\beta(\m_i)<\beta(\m)$. (The pairs $(\m_i,\theta_i)$ are not necessarily distinct.) Applying $^t$ and the induction hypothesis we get
\[
\std{\m}\sdp\theta^t=L(\m;\theta)^t+\sum_iZ(\m_i;\theta_i^t).
\]
Since $Z(\m;\theta^t)$ occurs in $\std{\m}\sdp\theta^t$ and $L(\m;\theta)^t$ is irreducible we necessarily have
$L(\m;\theta)^t=Z(\m;\theta^t)$ by the uniqueness part of Theorem \ref{thm: Zelclass}.

For any segment $\Delta$ define
\[
\pstv{\Delta}=\begin{cases}\Delta&\expo(\Delta)\ge0,\\\tilde\Delta&\text{otherwise.}\end{cases}
\]
Given any multisegment $\m=\Delta_1+\dots+\Delta_k$ and $\sigma\in\IrrctmpG$ we set
\[
\pstv{\m}=\pstv{\Delta_1}+\dots+\pstv{\Delta_k},\ \ \std{\m;\sigma}=\std{\pstv{\m}}\sdp\sigma,\ \ Z(\m;\sigma)=\soc(\std{\m;\sigma}).
\]
Note that $\std{\m;\sigma}$ is \SI\ if and only if $Z(\m;\sigma)$ is irreducible if and only if $Z(\m_{=0})\sdp\sigma$ is irreducible.

The following is an analogue of \cite[Proposition 1.3]{MR2504024}, which is obtained from it by applying $^t$.
\begin{proposition}
For any $\m,\m'\in\MS{\Seg}$ and $\theta\in\IrrctmpG$ we have
\[
Z(\m+\m';\theta)\le Z(\m)\sdp Z(\m';\theta).
\]
Thus, if $Z(\m)\sdp Z(\m';\theta)$ is irreducible then it is equal to $Z(\m+\m';\theta)$.
\end{proposition}

Specializing to the case $\m'=0$ we get
\begin{corollary} \label{cor: Tadanalog}
For any $\m\in\MS{\Seg}$ and $\theta\in\IrrctmpG$ we have $Z(\m;\theta)\le Z(\m)\sdp\theta$.
Thus, if $Z(\m)\sdp \theta$ is irreducible then it is equal to $Z(\m;\theta)$
and in particular, $Z(\m_{=0})\sdp\theta$ is irreducible.
\end{corollary}

\begin{theorem} \cite{MR1658535}\footnote{Once again, this is only stated for odd orthogonal and symplectic groups,
but the same argument works in general} \label{thm: oneseg}
Let $\Delta\in\Seg$ and $\sigma\in\CuspG$. Then the following conditions are equivalent.
\begin{enumerate}
\item $Z(\Delta)\sdp\sigma$ is irreducible.
\item $L(\Delta)\sdp\sigma$ is irreducible.
\item $\rho\sdp\sigma$ is irreducible for every $\rho\in\Delta$.
\end{enumerate}
\end{theorem}



\begin{corollary} \label{cor: tempcase}
Let $\m=\Delta_1+\dots+\Delta_k$ with $\expo(\Delta_i)=0$ for all $i$ and let $\sigma\in\CuspG$.
Let $\Sigma=\std{\m}\sdp\sigma=Z(\m)\sdp\sigma$ and $\Sigma'=\cstd{\m}\sdp\sigma=L(\m)\sdp\sigma$
Then the following conditions are equivalent.
\begin{enumerate}
\item \label{part: totindzele} $\Sigma$ is irreducible (and cotempered).
\item \label{part: indzele} $Z(\Delta_i)\sdp\sigma$ is irreducible for all $i$.
\item \label{part: totindL} $\Sigma'$ is irreducible (and tempered).
\item \label{part: indL} $L(\Delta_i)\sdp\sigma$ is irreducible for all $i$.
\item \label{part: nosupp} $\rho\sdp\sigma$ is irreducible for every $\rho\in\cup_i\Delta_i$.
\end{enumerate}
In this case, $\Sigma^\vee=\std{\m^\Galinv}\sdp\sigma^\vee$ and
$\Sigma'^\vee=\cstd{\m^\Galinv}\sdp\sigma^\vee$.
\end{corollary}

Indeed, the equivalence of conditions \ref{part: totindL} and \ref{part: indL} (for any $\sigma$ square-integrable) follows from \cite[Theorem 13.1]{MR1896238}
which is based on the results and techniques of Goldberg \cite{MR1296726, MR1224616, MR1335083}.
The equivalence of conditions \ref{part: totindzele} and \ref{part: totindL} follows from the properties of $^t$.
The rest follows from Theorem \ref{thm: oneseg} and \eqref{eq: indtildeinv}.

\begin{definition} \label{def: pi+}
For any multisegment $\m=\Delta_1+\dots+\Delta_k$ we write
\[
\m_{>0}=\sum_{i:\expo(\Delta_i)>0}\Delta_i,
\]
and similarly for $\m_{\ge0}$, $\m_{<0}$, $\m_{\le0}$, $\m_{=0}$.
Thus,
\[
\m=\m_{>0}+\m_{=0}+\m_{<0}=\m_{\ge0}+\m_{<0}=\m_{>0}+\m_{\le0}.
\]
If $\pi=Z(\m)$ then we write $\pi_+=Z(\m_{>0})$.
\end{definition}

The following consequence will be our main tool for proving irreducibility.

\begin{corollary} \label{cor: stdirredcrit}
\begin{enumerate}
\item Let $\m\in\MS{\Seg_{>0}}$, $\theta\in\IrrctmpG$ and $\pi\in\RepsG$. Assume that
\[
\pi\hookrightarrow\std{\m;\theta}
\]
and
\[
\pi^\vee\hookrightarrow\std{\m^\Galinv;\theta^\vee}.
\]
Then $\pi$ is irreducible (and isomorphic to $Z(\m;\theta)$).
\item Assume that $\m\in\MS{\Seg_{\ge0}}$ and $\sigma\in\CuspG$ are such that
$\rho\sdp\sigma$ is irreducible for every $\rho\in\supp\m_{=0}$.
Let $\pi\in\RepsG$ and assume that
\[
\pi\hookrightarrow \std{\m;\sigma}
\]
and
\[
\pi^\vee\hookrightarrow \std{\m^\Galinv;\sigma^\vee}.
\]
Then $\pi$ is irreducible (and isomorphic to $Z(\m;\sigma)=Z(\m_{>0};\std{\m_{=0}}\sdp\sigma)$).
\end{enumerate}
\end{corollary}

\begin{proof}
\begin{enumerate}
\item By Theorem \ref{thm: Zelclass} and \eqref{eq: obvious} both $\pi$ and $\pi^\vee$ are \SI\ and
\[
\soc(\pi)^\vee=Z(\m;\theta)^\vee=Z(\m^\Galinv;\theta^\vee)=\soc(\pi^\vee).
\]
Hence $\pi$ is irreducible by \eqref{eq: pisimple}.
\item By Corollary \ref{cor: tempcase} $\std{\m_{=0}}\sdp\sigma$ is irreducible.
The irreducibility of $\pi$ now follows from the first part.\qedhere
\end{enumerate}
\end{proof}

\section{A reducibility result} \label{sec: red}

For any $\sigma\in\IrrG$ we write
\[
\cuspred_\sigma=\{\rho\in\CuspGL:\rho\sdp\sigma\text{ is reducible}\}.
\]
By \eqref{eq: indtildeinv} we have $\tilde{\cuspred_\sigma}=\cuspred_\sigma$.

In this section we prove the following result.
\begin{theorem} \label{thm: mainred}
Suppose that $\sigma\in\CuspG$ and $\pi\in\IrrGL$ are such that $\supp\pi\cap \cuspred_\sigma\ne\emptyset$.
Then $\pi\sdp\sigma$ is reducible.
\end{theorem}

\subsection{}
We start with the following definition. Recall the notation of \S\ref{sec: derivatives}.
\begin{definition}
Let $\pi\in\IrrGL$ and $A\subset\CuspGL$.
\begin{enumerate}
\item We say that $\pi$ is \mnml{$A$} if $D_{A^c;A^c}(\pi)=\pi$, i.e., if $\lnrset(\pi)\cup\rnrset(\pi)\subset A$.
\item We say that $\pi$ is $A$-critical if it is \mnml{$A$}, different from $\one$, and for any $\rho\in A$, either $D_{\rho;\tilde\rho}(\pi)=\pi$
(i.e., $\pi$ is left $\rho$-reduced and right $\tilde\rho$-reduced) or $\supp D_{\rho;\tilde\rho}(\pi)\cap A=\emptyset$.
\end{enumerate}
\end{definition}

Note that $\pi$ is \mnml{$A$} (resp., $A$-critical) if and only if $\tilde\pi$ is \mnml{$\tilde A$} (resp., $\tilde A$-critical).

Next, we recall basic facts about the set $\cuspred_\sigma$ when $\sigma\in\CuspG$.
\begin{theorem}(Casselman, Silberger, M\oe glin) \label{thm: cuspred}
Let $\sigma\in\CuspG$ and $\rho\in\cuspred_\sigma$. Then
\begin{enumerate}
\item $\tilde\rho\in\rho[\R]:=\{\rho[x]:x\in\R\}$ (e.g., \cite[Th\'eor\`eme VII.1.3]{MR2567785}).
\item $\rho[\R]\cap\cuspred_\sigma=\{\rho,\tilde\rho\}$ \cite{MR577138}.
\item $\tilde\rho\in\rho[\Z]:=\{\rho[m]:m\in\Z\}$. Thus, if $\expo(\rho)\ge0$ then $[\tilde\rho,\rho]\in\Seg$
\cite[Th\'eor\`eme 3.1]{MR3220932}.\footnote{Technically, only quasi-split groups are considered in \cite{MR3220932} but as mentioned there,
this is only for simplicity.}
\end{enumerate}
\end{theorem}

The first two parts are special cases of general results about reductive groups.
The proof of the third part lies deeper -- it uses the stabilization of the twisted trace formula (at least in a simple form).
We will only use it in a superficial way to simplify the statements.

\begin{proposition} \label{prop: critclass}
Let $\sigma\in\CuspG$. Then the $\cuspred_\sigma$-critical representations are of the form
\begin{gather*}
Z([\tilde\alpha,\alpha])^{\times k}, L([\tilde\alpha,\alpha])^{\times k},\ k\ge1,\\
\alpha^{\times k}\times\tilde\alpha^{\times l},\ k,l\ge0, k+l>0\text{ (with $kl=0$ if $\tilde\alpha=\lshft{\alpha})$},\\
\text{(if $\tilde\alpha=\alpha$)  }\alpha^{\times k}\times Z([\alpha,\rshft{\alpha}]+[\lshft{\alpha},\alpha])^{\times l},\ k\ge0, l>0,
\end{gather*}
as we vary over $\alpha\in\cuspred_\sigma$ with $\expo(\alpha)\ge0$.
\end{proposition}

\begin{proof}
It is clear that the above representations are well-defined and $\cuspred_\sigma$-critical.
Conversely, suppose that $\pi=Z(\m)\in\IrrGL$ is $\cuspred_\sigma$-critical.
We first remark that since $\supp\pi\ne\emptyset$ and $\pi$ is \mnml{$\cuspred_\sigma$}, necessarily $\supp\pi\cap\cuspred_\sigma\ne\emptyset$.
By passing to $\tilde\pi$ if necessary we may assume that there exists $\alpha\in\supp\pi\cap\cuspred_\sigma$ with $\expo(\alpha)\ge0$.
We fix such $\alpha$ for the rest of the proof.

We first claim that $\supp\pi\subset\alpha[\Z]$.

For otherwise we could write $\pi=\pi_1\times\pi_2$ where $\pi_1,\pi_2\ne\one$, $\supp\pi_1\subset\alpha[\Z]$,
and $\supp\pi_2\cap\alpha[\Z]=\emptyset$.
If $\rho\in\lnrset(\pi_2)$ then $D_{\rho;\tilde\rho}(\pi)=\pi_1\times D_{\rho;\tilde\rho}(\pi_2)$.
Thus, $D_{\rho;\tilde\rho}(\pi)\ne\pi$ but $\alpha\in\supp D_{\rho;\tilde\rho}(\pi)$,
in contradiction to the assumption on $\pi$.

Next we claim that
\begin{equation} \label{step: bdelta<=}
b(\Delta)\le\alpha\text{ for any }\Delta\le\m.
\end{equation}

Indeed, let $\Delta\le\m$ with $b(\Delta)$ maximal. Then by Proposition \ref{prop: combprop} $b(\Delta)\in\lnrset(\pi)$ and hence,
$b(\Delta)\in\cuspred_\sigma\cap\alpha[\Z]=\{\tilde\alpha,\alpha\}$ by Theorem \ref{thm: cuspred}.

Similarly,
\begin{equation} \label{step: edelta>=}
e(\Delta)\ge\tilde\alpha\text{ for any }\Delta\le\m.
\end{equation}

In particular,
\begin{equation}\label{step: onlysingle}
\rho\in[\tilde\alpha,\alpha]\text{ for any }\{\rho\}\le\m.
\end{equation}

Henceforth, let $r$ be the maximal length of a segment in $\m$ and
for any segment $\Delta$ denote by $a_{\m}(\Delta)$ its multiplicity in $\m$.

Proposition \ref{prop: combprop} and the assumption that $\pi$ is \mnml{$\cuspred_\sigma$} imply that
for any segment $\Delta$ of length $r$ we have
\begin{subequations} \label{step: precsucc}
\begin{gather}
a_{\m}(\Delta)\le a_{\m}(\rshft\Delta)\text{ if }b(\Delta)\notin \cuspred_\sigma,\\
a_{\m}(\Delta)\le a_{\m}(\lshft\Delta)\text{ if }e(\Delta)\notin \cuspred_\sigma.
\end{gather}
\end{subequations}

\emph{From now on, let $\Delta$ be the segment in $\m$ of length $r$ with $b(\Delta)$ maximal.}
For $\rho=b(\Delta)$ we have
\begin{equation} \label{step: bdelta}
\rho\in\{\alpha,\tilde\alpha\}\text{ and }\alpha,\tilde\alpha\notin\supp D_{\rho;\tilde\rho}(\pi).
\end{equation}
Indeed, Proposition \ref{prop: combprop} $\rho\in\lnrset(\pi)\subset\cuspred_\sigma\cap\supp\pi=\{\alpha,\tilde\alpha\}$.
The second statement follows from the condition on $\pi$ since $D_{\rho;\tilde\rho}(\pi)\ne\pi$.

Consider now the case $\tilde\alpha=\alpha$.

\begin{enumerate}
\item If $r=1$, i.e., if all elements of $\m$ are singletons, then by \eqref{step: onlysingle} we have $\pi=\alpha^{\times k}$ for some $k>0$.

\item Suppose that $r=2$. Then by \eqref{step: bdelta}, $\Delta=[\alpha,\rshft\alpha]$
and by \eqref{step: precsucc} $a_{\m}(\lshft\Delta)=a_{\m}(\Delta)$.
Using \eqref{step: bdelta<=}, \eqref{step: edelta>=} and \eqref{step: onlysingle} we necessarily have
\[
\pi=\alpha^{\times k}\times Z([\alpha,\rshft{\alpha}]+[\lshft{\alpha},\alpha])^{\times l}
\]
for some $k\ge0$ and $l>0$.

\item Assume on the contrary that $r>2$. Then by \eqref{step: precsucc}, $\lshft\Delta\le\m$.
On the other hand, $\alpha=b(\Delta)\in\lshft{\Delta}$ but $\alpha\ne b(\lshft{\Delta}),e(\lshft{\Delta})$ since $r>2$.
By Corollary \ref{cor: rhoinsuppD} (applied to $\lshft{\Delta}$) we have $\alpha\in\supp D_{\alpha;\alpha}(\pi)$ in contradiction with \eqref{step: bdelta}.
\end{enumerate}

This concludes the case $\tilde\alpha=\alpha$.
\emph{From now on we assume that $\tilde\alpha\ne\alpha$.}

\begin{enumerate}
\item Suppose that $r=1$, i.e., all segments in $\m$ are singletons. Then either
\[
\pi=L([\tilde\alpha,\alpha])^{\times k}
\]
for some $k>0$ or
\[
\pi=\alpha^{\times k}\times\tilde\alpha^{\times l}
\]
for some $k,l\ge0$ with $k+l>0$ where $kl=0$ if $\tilde\alpha=\lshft\alpha$.

Indeed, by \eqref{step: onlysingle} and \eqref{step: precsucc} we have
\[
\m=k(\sum_{\rho\in[\tilde\alpha,\alpha]}\{\rho\})+l\{\alpha\}+m\{\tilde\alpha\}
\]
for some $k,l,m\ge0$. If $\tilde\alpha=\lshft{\alpha}$ we may assume that $l$ or $m$ equals $0$. Thus,
\[
\pi=L([\tilde\alpha,\alpha])^{\times k}\times\alpha^{\times l}\times\tilde\alpha^{\times m}.
\]
We have
\[
D_{\tilde\alpha;\alpha}(\pi)=L([\tilde\alpha,\alpha])^{\times k}
\]
and therefore either $k=0$ or $l=m=0$ for otherwise $\pi$ is not $\cuspred_\sigma$-critical.

\item Suppose that $r>1$. Then $\Delta=[\tilde\alpha,\alpha]$.

Assume on the contrary that $\Delta\ne[\tilde\alpha,\alpha]$. Then $e(\Delta)\not\in\cuspred_\sigma$
(since $b(\Delta)\in\{\alpha,\tilde\alpha\}$ and $\Delta\ne\{\alpha\},\{\tilde\alpha\}$)
and hence $\lshft{\Delta}\le\m$ by \eqref{step: precsucc}. Let $\rho=b(\Delta)$. Since $\rho\in\lshft\Delta$,
$\rho\ne b(\lshft{\Delta})$ and $\tilde\rho\ne\rho$ we have $\rho\in\supp D_{\rho;\tilde\rho}(\pi)$ by Corollary \ref{cor: rhoinsuppD}
in contradiction with \eqref{step: bdelta}.

\emph{For the rest of the proof we assume that $\Delta=[\tilde\alpha,\alpha]$ (and $\tilde\alpha\ne\alpha$).}

\item
We have
\begin{equation} \label{part: leftalpha}
D_{\alpha;\tilde\alpha}(\pi)=\pi,\text{i.e., $\alpha\notin\lnrset(\pi)$ and $\tilde\alpha\notin\rnrset(\pi)$,}
\end{equation}
and $\supp D_{\tilde\alpha;\alpha}(\pi)\cap \cuspred_\sigma=\emptyset$.

Indeed, by Corollary \ref{cor: rhoinsuppD} $\alpha\in\supp D_{\alpha;\tilde\alpha}(\pi)$.
Thus, $D_{\alpha;\tilde\alpha}(\pi)=\pi$ by the assumption on $\pi$.
On the other hand, $D_{\tilde\alpha;\alpha}(\pi)\ne\pi$ since $\tilde\alpha\in\lnrset(\pi)$ and therefore
$\supp D_{\tilde\alpha;\alpha}(\pi)\cap \cuspred_\sigma=\emptyset$ by the assumption on $\pi$.

\item Assume that there exists $\Delta'\ne\Delta$ with $\Delta'\le\m$ and take
such $\Delta'$ with $b(\Delta')$ maximal. Then either $b(\Delta')=\tilde\alpha$ or $b(\Delta')=\lshft{\tilde\alpha}$.

Otherwise, we would have $b(\Delta')\in\lnrset(\pi)$ (by Proposition \ref{prop: combprop} and the maximality of $b(\Delta')$)
and this contradicts either the assumption that $\pi$ is \mnml{$\cuspred_\sigma$} if $b(\Delta')\ne\alpha$ or \eqref{part: leftalpha} if $b(\Delta')=\alpha$.

\item The condition $b(\Delta')=\tilde\alpha$ implies $\Delta'=\{\tilde\alpha\}$.

Otherwise, $e(\Delta')\notin \cuspred_\sigma$ and therefore $e(\Delta')\notin\rnrset(\pi)$.
Thus, by Proposition \ref{prop: combprop} there exists $\Delta''\in\m$ such that $\Delta''\prec\Delta'$ and $e(\Delta'')=e(\lshft{\Delta'})$.
Hence, $b(\Delta'')\ne\tilde\alpha\in\Delta''$ and therefore by Corollary \ref{cor: rhoinsuppD}
$\tilde\alpha\in\supp D_{\tilde\alpha;\alpha}(\pi)$ which contradicts \eqref{part: leftalpha}.

\item $\Delta'\ne\{\tilde\alpha\}$.

Otherwise, by \eqref{step: edelta>=} and Proposition \ref{prop: combprop} $\tilde\alpha\in\rnrset(\pi)$
and we again get a contradiction to \eqref{part: leftalpha}.

\item $b(\Delta')\ne\lshft{\tilde\alpha}$.

Otherwise, $\tilde\alpha\in\Delta'$ (by \eqref{step: edelta>=}) and therefore by Corollary \ref{cor: rhoinsuppD}
$\tilde\alpha\in\supp D_{\tilde\alpha;\alpha}(\pi)$ and we get a contradiction to \eqref{part: leftalpha}.

\item Thus, $\pi=Z(\Delta)^{\times k}$ for some $k>0$.

\end{enumerate}
The proposition follows.
\end{proof}

\subsection{Proof of Theorem \ref{thm: mainred}}
Assume on the contrary that $\pi\sdp\sigma$ is irreducible for some $\pi\in\IrrGL$ such that $\supp\pi\cap \cuspred_\sigma\ne\emptyset$
and let $\pi$ be a counterexample with minimal $\deg\pi$. Clearly $\pi\ne\one$.
We claim that $\pi$ is $\cuspred_\sigma$-critical.
Indeed, $\pi$ is \mnml{$\cuspred_\sigma$} since otherwise $D_{\cuspred_\sigma^c;\cuspred_\sigma^c}(\pi)$ would be a counterexample
of smaller degree in view of \eqref{eq: suppDAB} and Corollary \ref{cor: redAAvee}. Further, for any $\rho\in\cuspred_\sigma$ we have either
$D_{\rho;\tilde\rho}(\pi)=\pi$ or $\supp D_{\rho;\tilde\rho}(\pi)\cap\cuspred_\sigma=\emptyset$ for otherwise $D_{\rho;\tilde\rho}(\pi)$
would be a counterexample of smaller degree.

By Proposition \ref{prop: critclass} and Theorem \ref{thm: oneseg} it remains to show the following lemma.
(It is clear that if $\pi_1\sdp\sigma$ is reducible then $\pi_1\times\pi_2\sdp\sigma$ is reducible for any $\pi_2$.)

\begin{lemma}
Suppose that $\tilde\alpha=\alpha\in\cuspred_\sigma$ and
\[
\pi=Z([\alpha,\rshft{\alpha}]+[\lshft{\alpha},\alpha]).
\]
Then $\pi\sdp\sigma$ is reducible.
\end{lemma}

\begin{proof}
Note that $\pi=L([\alpha,\rshft{\alpha}]+[\lshft{\alpha},\alpha])$.
By \cite[Proposition 1.3]{MR2504024} we have
\[
L([\alpha,\rshft{\alpha}]+[\alpha,\rshft{\alpha}];\sigma)\le\pi\sdp\sigma.
\]
On the other hand, clearly
\[
L([\alpha,\rshft{\alpha}]+[\alpha,\rshft{\alpha}];\sigma)\le L([\alpha,\rshft{\alpha}])\sdp L([\alpha,\rshft{\alpha}];\sigma).
\]
Thus, to prove the reducibility of $\pi\sdp\sigma$ it suffices to show that
\[
\pi\sdp\sigma\not\le L([\alpha,\rshft{\alpha}])\sdp L([\alpha,\rshft{\alpha}];\sigma).
\]
We show in fact that
\begin{equation} \label{eq: reqeq}
\comultG_{\{\alpha+\lshft{\alpha}\};*}(\pi\sdp\sigma)\not\le
\comultG_{\{\alpha+\lshft{\alpha}\};*}(L([\alpha,\rshft{\alpha}])\sdp L([\alpha,\rshft{\alpha}];\sigma)).
\end{equation}
By \eqref{eq: comodladder} and \eqref{eq: Tadic formula}, the left-hand side of \eqref{eq: reqeq} is
\[
2L([\lshft{\alpha},\alpha])\otimes L([\lshft{\alpha},\alpha])\sdp\sigma.
\]
On the other hand, we have
\[
\comod_{\MS{\{\alpha,\lshft{\alpha}\}};*}(L([\alpha,\rshft{\alpha}]))=L([\lshft{\alpha},\alpha])\otimes\one+
\alpha\otimes\rshft{\alpha}+\one\otimes L([\alpha,\rshft{\alpha}])
\]
and in particular
\[
\comod_{\{\lshft{\alpha}\};*}(L([\alpha,\rshft{\alpha}]))=0.
\]
Hence,
\[
\comultG_{\{\lshft{\alpha}\};*}(L([\alpha,\rshft{\alpha}];\sigma))\le
\comultG_{\{\lshft{\alpha}\};*}(L([\alpha,\rshft{\alpha}])\sdp\sigma)=0.
\]
Thus, the right-hand side of \eqref{eq: reqeq} is
\[
L([\lshft{\alpha},\alpha])\otimes\left(L([\lshft{\alpha},\alpha])\sdp\sigma+L([\alpha,\rshft{\alpha}];\sigma)\right)
\]
and \eqref{eq: reqeq} follows since $L([\lshft{\alpha},\alpha])\sdp\sigma$ is reducible.
\end{proof}
This concludes the proof of Theorem \ref{thm: mainred}.

\section{Case of two segments} \label{sec: twosegments}

Throughout this section we fix $\sigma\in\CuspG$. Thus,
\[
\comultG(\sigma)=\one\otimes\sigma.
\]
We will prove the following result.
\begin{theorem} \label{thm: twosegments}
Let $\Delta_1, \Delta_2$ be two segments such that $\Delta_i\cap\cuspred_\sigma=\emptyset$, $i=1,2$.
Then $Z(\Delta_1+\Delta_2)\sdp\sigma$ is reducible if and only if $\Delta_1$ and $\widetilde{\Delta_2}$ are linked
and at least one of the following conditions is satisfied:
\begin{enumerate}
\item $\Delta_1$ and $\Delta_2$ are unlinked.
\item $\expo(\Delta_1)\cdot\expo(\Delta_2)<0$.
\end{enumerate}
\end{theorem}

\begin{remark}
In the next section we will generalize this result.
Nevertheless, although not logically necessarily, we opted to first give a proof in the special case in order to illustrate the ideas.
\end{remark}

Suppose first that $\Delta_1$ and $\Delta_2$ are unlinked, i.e., that $Z(\Delta_1+\Delta_2)=Z(\Delta_1)\times Z(\Delta_2)$.
Then by Theorem \ref{thm: Zelclass} part \ref{part: stdirred}, $Z(\Delta_1+\Delta_2)\sdp\sigma$ is irreducible if and only if
$\Delta_1$ and $\widetilde{\Delta_2}$ are unlinked.
(Of course, for the ``only if'' direction we do not need the assumption $\Delta_i\cap\cuspred_\sigma=\emptyset$.)

Thus, for the rest of the proof we may assume that $\Delta_1$ and $\Delta_2$ are linked, and without loss of generality that
$\Delta_2\prec\Delta_1$.
In this case, we split Theorem \ref{thm: twosegments} to the following two assertions which will be proved below.

\begin{lemma} \label{lem: irredcrit}
Suppose that $\Delta_2\prec\Delta_1$ and $\Delta_i\cap\cuspred_\sigma=\emptyset$, $i=1,2$.
Suppose further that $\Delta_1$ and $\widetilde{\Delta_2}$ are unlinked or that $\expo(\Delta_1)\cdot\expo(\Delta_2)\ge0$.
Then $Z(\Delta_1+\Delta_2)\sdp\sigma$ is irreducible.
\end{lemma}

\begin{lemma} \label{lem: 2segred}
Suppose that $\Delta_2\prec\Delta_1$, $\Delta_1$ and $\widetilde{\Delta_2}$ are linked and $\expo(\Delta_1)>0>\expo(\Delta_2)$.
Then $Z(\Delta_1+\Delta_2)\sdp\sigma$ is reducible.
\end{lemma}

\subsection{Auxiliary results}
Denote by $\IrrgenGL\subset\IrrGL$ the subset of irreducible generic representations.
In terms of the Zelevinsky classifications, these are the irreducible representations corresponding to multisegments consisting
of singleton segments, i.e.,
\[
\IrrgenGL=\{Z(\m):\m\in\MS{\{\rho\}:\rho\in\CuspGL}\}.
\]
In terms of the Langlands classification, $\IrrgenGL$ corresponds to the multisegments consisting of pairwise unlinked segments
(i.e., such that $\cstd{\m}=L(\m)$).
Dually, we say that $\pi\in\IrrGL$ is \emph{cogeneric} if $\pi=\std{\m}$ for some $\m\in\MS{\Seg}$.\footnote{The cogeneric representations
with a non-trivial vector fixed under the Iwahori subgroup are exactly the unramified representations.}
Denote by $\IrrcogenGL$ the set of irreducible cogeneric representations.
The sets $\IrrgenGL$ and $\IrrcogenGL$ correspond under the Zelevinsky involution.

The following is well known.
\begin{lemma} \label{lem: prodgen}
Let $\pi_1,\pi_2\in\IrrGL$. Then $\pi\le\pi_1\times\pi_2$ for some $\pi\in\IrrgenGL$ if and only if $\pi_1,\pi_2\in\IrrgenGL$.
In this case, $\pi$ is uniquely determined by $\pi_1$ and $\pi_2$ and it occurs with multiplicity one in $\JH(\pi_1\times\pi_2)$.
An analogous statement holds for $\IrrcogenGL$ (by applying the Zelevinsky involution).
\end{lemma}

\begin{lemma} \label{lem: negexpo}
For any $\m\in\MS{\Seg}$ there exists $\pi\in\IrrcogenGL$ such that $\expo(\rho)\le0$ for all $\rho\in\supp\pi$ and
\[
\pi\le\comodmax(\std{\m}).
\]
Consequently,
\[
\pi\otimes\sigma\le\comultG(\std{\m}\sdp\sigma).
\]
(In fact, this holds for any $\sigma\in\RepsG$.)
\end{lemma}

\begin{proof}
Since $\comodmax$ is a homomorphism, it is enough by Lemma \ref{lem: prodgen} to consider the case $\m=\Delta\in\Seg$.
This case follows from the formula \eqref{eq: comodmaxseg}.
Indeed, if $\Delta=\{\rho_1,\dots,\rho_l\}$ with $\rho_{i+1}=\rshft{\rho_i}$, $i=1,\dots,l-1$ then we take
$\rho=\rho_i$ where $i$ is the largest index such that $\expo(\rho_i)\le0$ if $\expo(\rho_1)\le0$ and $\rho=\lshft{\rho_1}$ otherwise.
\end{proof}

\begin{lemma} \label{lem: posexpo}
Let $\Delta_1,\Delta_2\in\Seg$ be such that $\Delta_2\prec\Delta_1$ and $\expo(\Delta_2)\ge0$.
Suppose that
\[
\pi\otimes\sigma\le\comultG(Z(\Delta_1+\Delta_2)\sdp\sigma)
\]
for some $\pi\in\IrrcogenGL$. Then there exists $\rho\in\supp\pi$ such that $\expo(\rho)>0$.
\end{lemma}

\begin{proof}
By \eqref{eq: Tadic formula} and the supercuspidality of $\sigma$ we have $\pi\le\comodmax(Z(\Delta_1+\Delta_2))$.
By \eqref{eq: comodmaxladder} there exist $\rho_1\in\lextend\Delta_1$ and $\rho_2\in\lextend\Delta_2$ with $\rho_2<\rho_1$ such that
\[
\pi\le  Z(\m_1)\times Z(\tilde\m_2)
\]
where $\m_1=[b(\Delta_1),\rho_1]+[b(\Delta_2),\rho_2]$ and $\m_2=[\rshft{\rho_1},e(\Delta_1)]+[\rshft{\rho_2},e(\Delta_2)]$.
Thus, by Lemma \ref{lem: prodgen} $Z(\m_2)$ is cogeneric and therefore $\rho_1>e(\Delta_2)$.
Hence, $e(\rshft{\Delta_2})\in [b(\Delta_1),\rho_1]\subset\supp\pi$ (since $\Delta_2\prec\Delta_1$) while $\expo(e(\Delta_2))\ge\expo(\Delta_2)\ge0$.
\end{proof}

\subsection{Proof of Lemma \ref{lem: irredcrit}}
We show that
\begin{equation} \label{eq: embeds}
Z(\Delta_1+\Delta_2)\sdp\sigma\hookrightarrow\std{\Delta_1+\Delta_2;\sigma}
\end{equation}
and
\begin{equation} \label{eq: surjects}
Z(\Delta_1^\vee+\Delta_2^\vee)\sdp\sigma^\vee\hookrightarrow\std{\Delta_1^\Galinv+\Delta_2^\Galinv;\sigma^\vee}.
\end{equation}

We first consider the case where $\expo(\Delta_1)>0>\expo(\Delta_2)$ and $\Delta_1$, $\widetilde{\Delta_2}$ are unlinked.
In this case,
\[
Z(\Delta_1+\Delta_2)\sdp\sigma\hookrightarrow Z(\Delta_1)\times Z(\Delta_2)\sdp\sigma\simeq
Z(\Delta_1)\times Z(\widetilde{\Delta_2})\sdp\sigma=\std{\Delta_1+\Delta_2;\sigma}
\]
while
\begin{multline*}
Z(\Delta_1^\vee+\Delta_2^\vee)\sdp\sigma^\vee\hookrightarrow Z(\Delta_2^\vee)\times Z(\Delta_1^\vee)\sdp\sigma^\vee\simeq\\
Z(\Delta_2^\vee)\times Z(\Delta_1^\Galinv)\sdp\sigma^\vee\simeq Z(\Delta_1^\Galinv)\times Z(\Delta_2^\vee)\sdp\sigma^\vee=
\std{\Delta_1^\Galinv+\Delta_2^\Galinv;\sigma^\vee}
\end{multline*}
as required.

It remains to consider the case $\expo(\Delta_1)\cdot\expo(\Delta_2)\ge0$.
Without loss of generality we may assume that $\expo(\Delta_1)>\expo(\Delta_2)\ge0$ -- otherwise replace $(\Delta_1,\Delta_2)$
by $(\widetilde{\Delta_2},\widetilde{\Delta_1})$. We assume it for the rest of the proof.
Then $\std{\Delta_1+\Delta_2;\sigma}=Z(\Delta_1)\times Z(\Delta_2)\sdp\sigma$ and \eqref{eq: embeds} is clear.
It remains to prove \eqref{eq: surjects}.
We have
\begin{equation} \label{eq: 1stembed}
Z(\Delta_1^\vee+\Delta_2^\vee)\sdp\sigma^\vee\hookrightarrow Z(\Delta_2^\vee)\times Z(\Delta_1^\vee)\sdp\sigma^\vee\simeq\\
Z(\Delta_2^\vee)\times Z(\Delta_1^\Galinv)\sdp\sigma^\vee.
\end{equation}
As before, in the case where $\Delta_1$, $\widetilde{\Delta_2}$ are unlinked we have
\[
Z(\Delta_2^\vee)\times Z(\Delta_1^\Galinv)\sdp\sigma^\vee
\simeq Z(\Delta_1^\Galinv)\times Z(\Delta_2^\vee)\sdp\sigma^\vee\simeq
Z(\Delta_1^\Galinv)\times Z(\Delta_2^\Galinv)\sdp\sigma^\vee=\std{\Delta_1^\Galinv+\Delta_2^\Galinv;\sigma^\vee}.
\]
Thus, we may assume that $\Delta_1$, $\widetilde{\Delta_2}$ are linked, in which case
necessarily $\widetilde{\Delta_2}\prec\Delta_1$ (or equivalently, $\widetilde{\Delta_1}\prec\Delta_2$) since $\expo(\Delta_1)>0\ge\expo(\widetilde{\Delta_2})$.
We thus have an exact sequence
\begin{equation} \label{eq: es}
0\rightarrow \Pi:=Z(\Delta_1')\times Z(\Delta_2')\sdp\sigma^\vee\rightarrow Z(\Delta_2^\vee)\times Z(\Delta_1^\Galinv)\sdp\sigma^\vee\rightarrow
Z(\Delta_1^\Galinv+\Delta_2^\vee)\sdp\sigma^\vee\rightarrow0,
\end{equation}
where $\Delta_1'=\Delta_2^\vee\cup\Delta_1^\Galinv$ and $\Delta_2'=\Delta_2^\vee\cap\Delta_1^\Galinv\subset\Delta_1'$.
We claim that $\Pi$ is irreducible. Indeed, $\widetilde{\Delta_1'}=\Delta_2^\Galinv\cup\Delta_1^\vee=[e(\Delta_1)^\vee,e(\Delta_2^\Galinv)]$
and $\Delta_2'=[b(\Delta_1^\Galinv),b(\Delta_2)^\vee]$ since $\widetilde{\Delta_2}\prec\Delta_1$.
Thus, $\widetilde{\Delta_1'}\supset\Delta_2'$ (since $\expo(\Delta_1),\expo(\Delta_2)\ge 0$) and in particular $\widetilde{\Delta_1'}$ and $\Delta_2'$ are unlinked.
The irreducibility of $\Pi$ follows from Theorem \ref{thm: Zelclass} part \ref{part: stdirred}.

By Lemmas \ref{lem: negexpo} and \ref{lem: posexpo} we have
\[
\comultG(\Pi)\not\le\comultG(Z(\Delta_1^\Galinv+\Delta_2^\Galinv)\sdp\sigma^\vee)=\comultG(Z(\Delta_1^\vee+\Delta_2^\vee)\sdp\sigma^\vee)
\]
and therefore
\[
\Pi\not\le Z(\Delta_1^\vee+\Delta_2^\vee)\sdp\sigma^\vee.
\]
Hence, it follows from \eqref{eq: 1stembed} and \eqref{eq: es} that
\begin{multline*}
Z(\Delta_1^\vee+\Delta_2^\vee)\sdp\sigma^\vee\hookrightarrow Z(\Delta_1^\Galinv+\Delta_2^\vee)\sdp\sigma^\vee\hookrightarrow
Z(\Delta_1^\Galinv)\times Z(\Delta_2^\vee)\sdp\sigma^\vee\simeq Z(\Delta_1^\Galinv)\times Z(\Delta_2^\Galinv)\sdp\sigma^\vee\\
=\std{\Delta_1^\Galinv+\Delta_2^\Galinv;\sigma^\vee}
\end{multline*}
as required.

\subsection{Proof of Lemma \ref{lem: 2segred}}
By Theorem \ref{thm: mainred} we may assume that $\Delta_i\cap\cuspred_\sigma=\emptyset$, $i=1,2$.
Upon replacing $(\Delta_1,\Delta_2)$ by $(\widetilde{\Delta_2},\widetilde{\Delta_1})$ if necessary,
we may assume without loss of generality that $\Delta_1\prec\widetilde{\Delta_2}$. Then
\begin{multline*}
Z(\Delta_1+\Delta_2)\sdp\sigma\hookrightarrow Z(\Delta_1)\times Z(\Delta_2)\sdp\sigma\simeq
Z(\Delta_1)\times Z(\widetilde{\Delta_2})\sdp\sigma\hookrightarrow\\
\Pi:=Z(\Delta_1)\times Z(\Delta_1\cap\widetilde{\Delta_2})\times
Z(\widetilde{\Delta_2}\setminus\Delta_1)\sdp\sigma=
Z(\Delta_1\cap\widetilde{\Delta_2})\times Z(\Delta_1)\times
Z(\widetilde{\Delta_2}\setminus\Delta_1)\sdp\sigma.
\end{multline*}
Since $\expo(\Delta_1)>0$ and $\Delta_2\prec\Delta_1$, it is easy to see that
\[
\comod_{\{\data_{\Delta_1}+\data_{\Delta_1\cap\widetilde{\Delta_2}}\}}
(Z(\Delta_1)\times Z(\Delta_1\cap\widetilde{\Delta_2})\times Z(\widetilde{\Delta_2}\setminus\Delta_1))=
Z(\Delta_1)\times Z(\Delta_1\cap\widetilde{\Delta_2})\otimes Z(\widetilde{\Delta_2}\setminus\Delta_1).
\]
Hence the irreducible representation
\[
Z(\Delta_1)\times Z(\Delta_1\cap\widetilde{\Delta_2})\otimes Z(\widetilde{\Delta_2}\setminus\Delta_1)\sdp\sigma
\]
occurs with multiplicity one in $\comultG(\Pi)$.
Thus, $\Pi$ is \SI\ by Lemma \ref{lem: simpleSIcrit} and therefore $Z(\Delta_1+\Delta_2)\sdp\sigma$ is \SI\ and $\soc(Z(\Delta_1+\Delta_2)\sdp\sigma)=\soc(\Pi)$.
On the other hand,
\begin{multline*}
\soc(\Pi)=\soc(Z(\Delta_1)\times Z(\widetilde{\Delta_2})\sdp\sigma)=
\soc(Z(\Delta_1\cup\widetilde{\Delta_2})\times Z(\Delta_1\cap\widetilde{\Delta_2})\sdp\sigma)\\=\soc(Z(\Delta'_1)\times Z(\Delta'_2)\sdp\sigma)
\end{multline*}
where we write $\Delta'_1=\Delta_1\cup\widetilde{\Delta_2}$, $\Delta'_2=\widetilde{\Delta_1}\cap\Delta_2$.
Note that $\Delta'_2\prec\Delta'_1$ since $\expo(\Delta_1)>0>\expo(\Delta_2)$ and $\Delta_2\prec\Delta_1$.
Thus,
\[
\soc(Z(\Delta_1+\Delta_2)\sdp\sigma)=\soc(\Pi)=\soc(Z(\Delta'_1+\Delta'_2)\sdp\sigma).
\]
(In fact, $Z(\Delta'_1+\Delta'_2)\sdp\sigma$ is irreducible by Lemma \ref{lem: irredcrit} since
$\Delta'_1$ contains (and in particular, is unlinked with) $\widetilde{\Delta_2'}$. However, we will not need to use this fact.)

Suppose on the contrary that $Z(\Delta_1+\Delta_2)\sdp\sigma$ is irreducible. Then
\begin{equation} \label{eq: contra<=}
Z(\Delta_1+\Delta_2)\sdp\sigma\le Z(\Delta'_1+\Delta'_2)\sdp\sigma.
\end{equation}
We will show that this is impossible, arriving at a contradiction.

Note that $\Delta'_1=[b(\Delta_1),e(\widetilde{\Delta_2})]$, $\Delta'_2=[b(\widetilde{\Delta_1}),e(\Delta_2)]$ and
$\Delta'_1\cap\Delta'_2=[b(\Delta_1),e(\Delta_2)]=\Delta_1\cap\Delta_2$.
Let
\begin{gather*}
\text{$l_i$ (resp., $l_i'$) be the size of $\Delta_i$ (resp., $\Delta'_i$), $i=1,2$},\\
\text{$d=\deg\rho$ for any $\rho\in\Delta_1\cup\Delta_2$},\\
\text{$m=$ the size of $\Delta_1\cap\Delta_2=\Delta'_1\cap\Delta'_2$, $N=l_1+l_2=l'_1+l'_2$,}
\end{gather*}
Finally, denote by $\JacGLnm_{\min}$ the Jacquet functor with respect to the standard parabolic subgroup of type $(d,\dots,d)$ ($N$ times). Then
\[
\ell(\JacGLnm_{\min}(Z(\Delta_1+\Delta_2)))={N\choose l_1}-{N\choose m},\ \
\ell(\JacGLnm_{\min}(Z(\Delta'_1+\Delta'_2)))={N\choose l'_1}-{N\choose m}.
\]
Note that $l_1,l_2<l_1'$ (since $\Delta_1'\supsetneq\Delta_1,\widetilde{\Delta_2}$) and therefore
\[
\ell(\JacGLnm_{\min}(Z(\Delta'_1,\Delta'_2)))<\ell(\JacGLnm_{\min}(Z(\Delta_1+\Delta_2))).
\]
Denoting by $\JacG_{\min}$ the Jacquet functor with respect to the standard parabolic subgroup with Levi
part $\overbrace {\GL_d\times\dots\times\GL_d}^N\times G_{\deg\sigma}$.
It is easy to see by the geometric lemma that
$\ell(\JacG_{\min}(\pi\sdp\sigma))=2^N\ell(\JacGLnm_{\min}(\pi))$ for any $\pi\in\Reps(\GL_{dN})_{\MS{\Cusp\GL_d}}$.
Thus,
\begin{multline*}
\ell(\JacG_{\min}(Z(\Delta_1'+\Delta_2')\sdp\sigma))=2^N\ell(\JacGLnm_{\min}(Z(\Delta_1'+\Delta_2')))<\\
2^N\ell(\JacGLnm_{\min}(Z(\Delta_1+\Delta_2)))=\ell(\JacG_{\min}(Z(\Delta_1+\Delta_2)\sdp\sigma))
\end{multline*}
in contradiction to \eqref{eq: contra<=}.
Lemma \ref{lem: 2segred} follows.

This concludes the proof of Theorem \ref{thm: twosegments}.

\section{An irreducibility criterion} \label{sec: irred}

\subsection{}

In this subsection we reduce the question of irreducibility of $\pi\sdp\sigma$ where $\pi\in\Irr$ and $\sigma\in\CuspG$ to the case where
$\supp\pi\subset\rho[\Z]$ for some $\rho\in\CuspGL$ with $\tilde\rho\subset\rho[\Z]$, assuming knowledge of irreducibility of parabolic
induction for the general linear group.
For convenience and as a preparation for the subsequent subsections, we give a self-contained argument although some of the results are available
in greater generality in the literature.

\begin{lemma} \label{lem: badrhopos}
Let $\rho\in\CuspGL$ with $\tilde\rho\notin\rho[\Z]$ and let $\pi\in\IrrGL$ with $\supp\pi\subset\rho[\Z]$.
Assume that $\pi$ is \Zpstv.
Then $\pi\sdp\sigma$ is irreducible for any $\sigma\in\CuspG$.
\end{lemma}


\begin{proof}
We prove it by induction on $\deg\pi$. The case $\pi=\one$ is trivial.
For the induction step write $\pi=Z(\m)$ and $\m=\Delta+\m'$ where $\expo(\Delta)\ge\expo(\Delta')$
for any $\Delta'\le\m'$.
Since $\m$ is positive, $\pi\sdp\sigma\hookrightarrow\std{\m;\sigma}$.
On the other hand, by the assumption on $\rho$ and $\pi$, $Z(\Delta^\vee)\sdp\sigma^\vee$
and $Z(\m'^\vee)\times Z(\Delta^\Galinv)$ are irreducible.
Thus, by the induction hypothesis
\begin{multline*}
\pi^\vee\sdp\sigma^\vee\hookrightarrow Z(\m'^\vee)\times Z(\Delta^\vee)\sdp\sigma^\vee\simeq
Z(\m'^\vee)\times Z(\Delta^\Galinv)\sdp\sigma^\vee\simeq\\
Z(\Delta^\Galinv)\times Z(\m'^\vee)\sdp\sigma^\vee\simeq Z(\Delta^\Galinv)\times Z(\m'^\Galinv)\sdp\sigma^\vee=
\std{\m^\Galinv;\sigma^\vee}
\end{multline*}
and the lemma follows from Corollary \ref{cor: stdirredcrit}.
\end{proof}

\begin{lemma} \label{lem: nonselfdual}
Let $\rho\in\CuspGL$ with $\tilde\rho\notin\rho[\Z]$ and let $\pi\in\IrrGL$ with $\supp\pi\subset\rho[\Z]$.
Then $\pi\sdp\sigma$ is irreducible for any $\sigma\in\CuspG$.
\end{lemma}

\begin{proof}
Write $\pi=Z(\m)$. Clearly, $Z(\m_{=0})\sdp\sigma$ is irreducible.
The condition on $\rho$ implies that $\pi_1\times\pi_2$ is irreducible for any $\pi_1,\pi_2\in\IrrGL$
such that $\supp\pi_1\subset\rho[\Z]$ and $\supp\pi_2\subset\tilde\rho[\Z]$. In particular,
$\pi_+\times\tilde\pi_+$ and $\tilde\pi_+\times Z(\m_{=0})$ are irreducible. (Recall that by our convention
$\tilde\pi_+$ means $(\tilde\pi)_+$.) Thus, by Lemma \ref{lem: badrhopos}
\begin{multline*}
\pi\sdp\sigma\hookrightarrow\pi_+\times Z(\m_{=0})\times\pi_-\sdp\sigma\simeq
\pi_+\times Z(\m_{=0})\times\tilde\pi_+\sdp\sigma\simeq\\
\pi_+\times\tilde\pi_+\times Z(\m_{=0})\sdp\sigma=
Z(\m_{>0}+\tilde\m_{>0})\times Z(\m_{=0})\sdp\sigma\hookrightarrow\std{\m;\sigma}.
\end{multline*}
In a similar vein,
\[
\pi^\vee\sdp\sigma^\vee\hookrightarrow\std{\m^\Galinv;\sigma}
\]
and the lemma follows from Corollary \ref{cor: stdirredcrit}.
\end{proof}

\begin{remark}
Lemma \ref{lem: nonselfdual} is also a consequence of \cite[Proposition 3.2]{MR2504024}.
\end{remark}

\begin{lemma} \label{lem: redtok=2}
Let $\pi_i\in\IrrGL$, $i=1,2$ and $\sigma\in\IrrctmpG$. Assume that for any $\rho_i\in\supp\pi_i$, $i=1,2$ we have
$\rho_1,\tilde\rho_1\notin\rho_2[\Z]$. Then $\pi_1\times\pi_2\sdp\sigma$ is irreducible if and only if
$\pi_i\sdp\sigma$ is irreducible, $i=1,2$.
\end{lemma}

\begin{proof}
The condition is clearly necessary (for any $\sigma\in\IrrG$). Suppose that it is satisfied.
Write $\pi_i=Z(\m_i)$, $i=1,2$ so that $\pi=Z(\m)$ where $\m=\m_1+\m_2$.
Then
\begin{multline*}
\pi\sdp\sigma=\pi_1\times\pi_2\sdp\sigma\hookrightarrow
\pi_1\times Z(\pstv{\m_2})\sdp\sigma=Z(\pstv{\m_2})\times\pi_1\sdp\sigma\hookrightarrow\\ Z(\pstv{\m_2})\times Z(\pstv{\m_1})\sdp\sigma=
Z(\pstv{\m_2}+\pstv{\m_1})\sdp\sigma\hookrightarrow\std{\m;\sigma}.
\end{multline*}
Similarly,
\[
\pi^\vee\sdp\sigma^\vee\hookrightarrow\std{\m^\Galinv;\sigma^\vee},
\]
and the lemma follows from Corollary \ref{cor: stdirredcrit}.
\end{proof}

\begin{remark}
In fact Lemma \ref{lem: redtok=2} holds for any $\sigma\in\IrrG$.
This easily follows from \cite[Theorem 10.5]{MR1481814}.
\end{remark}

Recall that any $\pi\in\IrrGL$ can be written uniquely (up to permutation) as
\begin{equation} \label{eq: cusplinedecomp}
\pi=\pi_1\times\dots\times\pi_k
\end{equation}
where for all $i$, $\pi_i\ne\one$ and there exists $\rho_i\in\CuspGL$ such that $\supp\pi_i\subset\rho_i[\Z]$
and $\rho_j\notin\rho_i[\Z]$ for all $i\ne j$.
\begin{proposition}
Let $\pi\in\IrrGL$ and $\sigma\in\CuspG$. Write $\pi=\pi_1\times\dots\times\pi_k$ as in \eqref{eq: cusplinedecomp}.
Then $\pi\sdp\sigma$ is irreducible if and only if $\pi_i\sdp\sigma$ is irreducible for all $i$ such that $\tilde\rho_i\in\rho_i[\Z]$
and $\pi_i\times\tilde\pi_j$ is irreducible for all $i\ne j$ such that $\tilde\rho_j\in\rho_i[\Z]$.
\end{proposition}

\begin{proof}
By Lemma \ref{lem: redtok=2} we reduces the statement to the case where $k=2$ and $\tilde\rho_2\in\rho_1[\Z]$.
This case follows from Lemma \ref{lem: nonselfdual} since the irreducibility of $\pi\sdp\sigma=\pi_1\times\pi_2\sdp\sigma$ is equivalent to the irreducibility of
$\pi_1\times\tilde\pi_2\sdp\sigma$.
\end{proof}



\subsection{}
\begin{lemma} \label{lem: SI}
Assume that $\theta\in\IrrctmpG$ and $\pi\le\std{\m}$ where $\m$ is a positive multisegment.
Then
\[
\comultG_{\data_\m;*}(\pi\sdp\theta)=\pi\otimes\theta.
\]
Consequently, if $\pi$ is \SI\ then $\pi\sdp\theta$ is \SI\ and $\soc(\pi\sdp\theta)=\soc(\soc(\pi)\sdp\theta)$.
\end{lemma}

\begin{proof}
Define
\[
\comodrev:\GrGL\rightarrow\GrGL\otimes\GrGL
\]
by
\[
\comodrev=\comod-\id\otimes\one.
\]
Clearly $\comodrev(\tau)\ge0$ for all $\tau\in\RepsGL$.
We claim that for any irreducible $\alpha\otimes\beta\le \comodrev(\pi)$
we have $\expo(\alpha)<\expo(\pi)$. In fact, this is true for any irreducible $\alpha\otimes\beta\le \comod_{\ex}(\std{\m})$.
Indeed, by the multiplicativity of $\comod$ and the positivity of $\m$, it is enough to consider the case where $\m=\Delta$ with $\expo(\Delta)>0$.
We use formula \eqref{eq: comodseg}. For any $\Delta\le\m$ and $\rho,\rho'\in\lextend\Delta$ with $b(\Delta)\ne\rho\le\rho'$ we have
\[
\expo(Z([b(\Delta),\rho]))<\expo(\Delta)\text{ and }\expo(\tilde Z([\rshft{\rho'},e(\Delta)]))=
-\expo([\rshft{\rho'},e(\Delta)])\le 0<\expo(\Delta).
\]
Hence
\[
\expo(Z([b(\Delta),\rho])\times\tilde Z([\rshft{\rho'},e(\Delta)]))<\expo(\Delta).
\]
Our claim follows.

On the other hand, since $\theta$ is cotempered, for any irreducible $\alpha\otimes\beta\le\comultG(\theta)$ we have $\expo(\alpha)\le0$.
Since
\[
\comultG(\pi\sdp\theta)-\pi\otimes\theta=\comod(\pi)\sdp\comultG(\theta)-\pi\otimes\theta=
\comod_{\ex}(\pi)\sdp\comultG(\theta)+(\pi\otimes\one)\sdp(\comultG(\theta)-\one\otimes\theta)
\]
it follows that for any irreducible $\alpha\otimes\beta\le\comultG(\pi\sdp\theta)-\pi\otimes\theta$ we have $\expo(\alpha)<\expo(\pi)$.
This clearly implies the first part of the lemma. The second part follows from lemma \ref{lem: simpleSIcrit}.
\end{proof}

\begin{lemma} \label{lem: posremains}
Suppose that $\pi_1,\pi_2\in\IrrGL$ are \Zpstv\ and $\pi\in\IrrGL$ is such that $\pi\le\pi_1\times\pi_2$.
Then $\pi$ is \Zpstv.
\end{lemma}

\begin{proof}
Let $\pi_i=Z(\m_i)$, $i=1,2$ where $\m_1$ and $\m_2$ are positive multisegments.
Then $\pi=Z(\m)$ where $\m$ is smaller than or equal to $\m_1+\m_2$ with respect to the partial order on multisegments introduced by Zelevinsky in \cite[\S7]{MR584084}.
It remains to observe that if $\n$ is positive then the same is true for any multisegments less than or equal to it. This follows
immediately from \eqref{eq: expoge} and the definition of the partial order on multisegments.
\end{proof}

\begin{lemma} \label{lem: SIstages}
Let $\m$ be a multisegment, $\pi=Z(\m)$ and $\sigma\in\IrrctmpG$.
Suppose that $Z(\m_{\le0})\sdp\sigma$ is irreducible and $\pi_+\times\tilde\pi_+$ is \SI.
Write $\soc(\pi_+\times\tilde\pi_+)=Z(\n)$.
Then $\pi\sdp\sigma$ is \SI\ and $\soc(\pi\sdp\sigma)=Z(\n;Z(\m_{=0})\sdp\sigma)$.
\end{lemma}

\begin{proof}
By Corollary \ref{cor: Tadanalog} and our assumption, $Z(\m_{=0})\sdp\sigma$ is irreducible,
and hence $Z(\m_{=0})\sdp\sigma=Z(\tilde\m_{=0})\sdp\sigma$.
We have
\[
\pi\sdp\sigma\hookrightarrow\pi_+\times Z(\m_{\le 0})\sdp\sigma
\simeq\pi_+\times Z(\tilde\m_{\ge 0})\sdp\sigma\hookrightarrow
\pi_+\times\tilde\pi_+\times Z(\m_{=0})\sdp\sigma.
\]
By Lemma \ref{lem: SI} we infer that $\pi\sdp\sigma$ is \SI\ and
\begin{multline*}
\soc(\pi\sdp\sigma)=\soc(\pi_+\times\tilde\pi_+\times Z(\m_{=0})\sdp\sigma)\\
\simeq\soc(Z(\n)\times Z(\m_{=0})\sdp\sigma)=\soc\std{\n;Z(\m_{=0})\sdp\sigma}.
\end{multline*}
(Note that $\n$ is positive by Lemma \ref{lem: posremains}.)
The proposition follows.
\end{proof}

\begin{lemma}(cf.~\cite[p.~173]{MR863522}) \label{lem: MW}
Let $\pi_1,\pi_2\in\IrrGL$. Suppose that at least one of $\pi_1\times\pi_2$ or $\pi_2\times\pi_1$ is \SI.
Then the following conditions are equivalent.
\begin{enumerate}
\item $\pi_1\times\pi_2$ is irreducible.
\item $\pi_1\times\pi_2\simeq\pi_2\times\pi_1$.
\item $\soc(\pi_1\times\pi_2)\simeq\soc(\pi_2\times\pi_1)$.
\end{enumerate}
\end{lemma}

\begin{proof}
Clearly, 1$\implies$2$\implies$3.
Let $\Pi=\pi_1\times\pi_2$. Interchanging $\pi_1$ and $\pi_2$ if necessary we may assume that $\pi_2\times\pi_1$ is \SI.
Applying the functor $\iota$ and using \eqref{eq: iotaind} we deduce that
\[
\Pi^\vee=\pi_1^\vee\times\pi_2^\vee\simeq\iota(\pi_1)\times\iota(\pi_2)=\iota(\pi_2\times\pi_1)
\]
is \SI\ and
\[
\soc(\Pi^\vee)=\soc(\iota(\pi_2\times\pi_1))=\iota(\soc(\pi_2\times\pi_1))=(\soc(\pi_2\times\pi_1))^\vee.
\]
Therefore, condition 3 is equivalent to $\soc(\Pi)^\vee\simeq\soc(\Pi^\vee)$
which in turn is equivalent to the irreducibility of $\Pi$ by \eqref{eq: pisimple}.
\end{proof}

\begin{corollary} \label{cor: irrconseq}
Let $\m$ be a multisegment, $\pi=Z(\m)$ and $\sigma\in\IrrctmpG$.
Suppose that $Z(\m_{\ge0})\sdp\sigma$ and $Z(\m_{\le0})\sdp\sigma$ are irreducible
and that both $\pi_+\times\tilde\pi_+$ and $\tilde\pi_+\times\pi_+$ are \SI.
Then the following conditions are equivalent.
\begin{enumerate}
\item \label{part: irred} $\pi\sdp\sigma$ is irreducible.
\item \label{part: tildesym} $\pi\sdp\sigma\simeq\tilde\pi\sdp\sigma$.
\item \label{part: soctildesym} $\soc(\pi\sdp\sigma)\simeq\soc(\tilde\pi\sdp\sigma)$.
\item \label{part: m>0m<0} $\pi_+\times\tilde\pi_+$ is irreducible.
\item \label{part: m>0m<0sym} $\pi_+\times\tilde\pi_+\simeq\tilde\pi_+\times\pi_+$.
\item \label{part: m>0m<0soc} $\soc(\pi_+\times\tilde\pi_+)\simeq\soc(\tilde\pi_+\times\pi_+)$.
\end{enumerate}
\end{corollary}

\begin{proof}
By Lemma \ref{lem: MW} the conditions \ref{part: m>0m<0}, \ref{part: m>0m<0sym} and \ref{part: m>0m<0soc} are equivalent.
Clearly, \ref{part: irred}$\implies$\ref{part: tildesym}$\implies$\ref{part: soctildesym}.
Let $\Pi=\pi\sdp\sigma$.
By Lemma \ref{lem: SIstages} (applied to both $\m$ and $\tilde\m$) and the uniqueness part of Theorem \ref{thm: Zelclass},
the conditions \ref{part: soctildesym} and \ref{part: m>0m<0soc} are equivalent.
Similarly, applying Lemma \ref{lem: SIstages} to both $\m$ and $\m^\vee$ (with $\sigma$ and $\sigma^\vee$ respectively)
we deduce that $\Pi$ and $\Pi^\vee$ are \SI\ and that the condition
\[
\soc(\Pi)^\vee\simeq\soc(\Pi^\vee)
\]
is equivalent to condition \ref{part: m>0m<0soc} in view of Theorem \ref{thm: Zelclass} part \ref{part: Zvee}.
It remains to apply \eqref{eq: pisimple}.
\end{proof}

\subsection{}

From now on we fix $\sigma\in\CuspG$.

\begin{lemma} \label{lem: nonineq11}
Suppose that $\m$ is a ladder as in \eqref{eq: decep} with $k>1$ and $\expo(\Delta_k)\ge0$.
Assume that $\widetilde{\Delta_1}\prec\Delta_k$. Then
\[
Z(\widetilde{\Delta_1}\setminus\Delta_k)\otimes\tau\not\le\comod(Z(\m))\text{ for any }\tau\in\IrrGL.
\]
Hence,
\[
Z(\widetilde{\Delta_1}\setminus\Delta_k)\otimes\tau\not\le\comultG(Z(\m)\sdp\sigma)\text{ for any }\tau\in\IrrG.
\]
\end{lemma}

\begin{proof}
The second statement follows from the first one since $\sigma$ is supercuspidal.
To prove the first part, let
\[
\data=\data_{\widetilde{\Delta_1}\setminus\Delta_k}=\sum_{\rho\in\widetilde{\Delta_1}\setminus\Delta_k}\rho\in\MS{\CuspGL}.
\]
Since $k>1$ it suffices to show that $\comod_{\{\data\};*}(Z(\m))$ is equal to
\[
\tilde Z(\big(\sum_{i=1}^{k-1}(\Delta_i\setminus\Delta_{i+1})\big)+(\Delta_k\setminus\widetilde{\Delta_k}))\otimes
Z(\sum_{i=1}^{k-1}[b(\Delta_i),e(\Delta_{i+1})]+[b(\Delta_k),e(\widetilde{\Delta_k})])
\]
if $\Delta_i\prec\Delta_{i-1}$ for all $i<k$ and $\widetilde{\Delta_k}\ne\Delta_k$ and $\comod_{\{\data\};*}(Z(\m))=0$ otherwise.
Consider a term $\tau_1\otimes\tau_2$ on the right-hand side of \eqref{eq: comodladder} where
\[
\tau_1=Z(\sum_{i=1}^k[b(\Delta_i),\rho_i])\times\tilde Z(\sum_{i=1}^k[\rshft{\rho_i'},e(\Delta_i)]),\ \ \tau_2= Z(\sum_{i=1}^k[\rshft{\rho_i},\rho_i'])
\]
with $\rho_i,\rho_i'\in\lextend\Delta_i$, $\rho_i'\ge\rho_i$ for all $i$, $\rho_1>\dots>\rho_k$,
and $\rho'_1>\dots>\rho'_k$. If $\data_{\tau_1}=\data$ then we would necessarily have $\rho_i=b(\lshft{\Delta_i})$ for all $i$,
$\Delta_{i+1}\prec\Delta_i$, $\rho'_i=e(\Delta_{i+1})$, $i=1,\dots,k-1$, $\rho'_k=e(\widetilde{\Delta_k})<e(\Delta_k)$.
Our assertion follows.
\end{proof}

\begin{lemma} \label{lem: nonegirr}
Suppose that $\m$ is a ladder as in \eqref{eq: decep} with $\expo(\Delta_k)\ge0$.
Assume that $\supp\m\cap\cuspred_\sigma=\emptyset$. Then $Z(\m)\sdp\sigma$ is irreducible.
\end{lemma}

\begin{proof}
We prove this by induction on $k$. The case $k=1$ is Theorem \ref{thm: oneseg}.
Assume that $k>1$ and that the statement holds for $k-1$. Clearly,
\[
Z(\m)\sdp\sigma\hookrightarrow Z(\Delta_1)\times\dots\times Z(\Delta_k)\sdp\sigma=
\std{\m;\sigma}.
\]
By Corollary \ref{cor: stdirredcrit} it suffices to show that
\[
Z(\m^\vee)\sdp\sigma^\vee\hookrightarrow Z(\Delta_1^\Galinv)\times\dots\times Z(\Delta_k^\Galinv)\sdp\sigma^\vee=
\std{\m^\Galinv;\sigma^\vee}.
\]
Write $\m'=\Delta_2+\dots+\Delta_k$ so that $\m=\m'+\Delta_1$.
We have
\begin{equation} \label{eq: 1st embed}
Z(\m^\vee)\sdp\sigma^\vee\hookrightarrow Z(\m'^\vee)\times Z(\Delta_1^\vee)\sdp\sigma^\vee\simeq
Z(\m'^\vee)\times Z(\Delta_1^\Galinv)\sdp\sigma^\vee.
\end{equation}
If $\widetilde{\Delta_k}\not\prec\Delta_1$, or equivalently, $\Delta_k^\vee\not\prec\Delta_1^\Galinv$
(in which case $\Delta_i^\vee\not\prec\Delta_1^\Galinv$ for all $i>1$) then
\[
Z(\m'^\vee)\times Z(\Delta_1^\Galinv)\sdp\sigma^\vee\simeq Z(\Delta_1^\Galinv)\times Z(\m'^\vee)\sdp\sigma^\vee
\]
which by the induction hypothesis is isomorphic to
\[
Z(\Delta_1^\Galinv)\times Z(\m'^\Galinv)\sdp\sigma^\vee\hookrightarrow\std{\m^\Galinv;\sigma^\vee}
\]
as required.

Assume now that $\widetilde{\Delta_k}\prec\Delta_1$ and let $\Gamma=\Delta_1\cup\widetilde{\Delta_k}$,
$\gamma=\Delta_1\cap\widetilde{\Delta_k}$, $\m''=\Delta_2+\dots+\Delta_{k-1}$ and $\n=\m''^\vee+\Gamma^\Galinv+\gamma^\Galinv$.
Then by Corollary \ref{cor: length2} we have an exact sequence
\begin{equation} \label{eq: es2}
0\rightarrow\Pi:=Z(\n)\sdp\sigma^\vee\rightarrow Z(\m'^\vee)\times Z(\Delta_1^\Galinv)\sdp\sigma^\vee\rightarrow
Z(\Delta_1^\Galinv+\m'^\vee)\sdp\sigma^\vee\rightarrow0.
\end{equation}
We first show that $\Pi$ is irreducible.
Note that $\Gamma=[b(\widetilde{\Delta_k}),e(\Delta_1)]$ and $\gamma=[b(\Delta_1),e(\widetilde{\Delta_k})]$.
Hence, $\expo(\Gamma),\expo(\gamma)\ge0$, $\Gamma\supset\tilde\gamma$ and $\Gamma\supset\Delta_i\supset\gamma$ for all $1<i<k$. Thus,
\begin{multline*}
\std{\m''}\times Z(\tilde\Gamma)\times Z(\tilde\gamma)\sdp\sigma\simeq
\std{\m''}\times Z(\tilde\Gamma)\times Z(\gamma)\sdp\sigma\simeq\\
\std{\m''}\times Z(\gamma)\times Z(\tilde\Gamma)\sdp\sigma\simeq
\std{\m''}\times Z(\gamma)\times Z(\Gamma)\sdp\sigma\simeq\std{\n^\Galinv;\sigma}
\end{multline*}
and therefore
\[
\Pi^\vee=Z(\n^\vee)\sdp\sigma\hookrightarrow Z(\m'')\times Z(\tilde\Gamma)\times Z(\tilde\gamma)\sdp\sigma\hookrightarrow
\std{\n^\Galinv;\sigma}.
\]
On the other hand,
\[
\Pi\hookrightarrow Z(\Gamma^\Galinv)\times Z(\gamma^\Galinv)\times Z(\m''^\vee)\sdp\sigma^\vee
\]
which by the induction hypothesis is isomorphic to
\[
Z(\Gamma^\Galinv)\times Z(\gamma^\Galinv)\times Z(\m''^\Galinv)\sdp\sigma^\vee\hookrightarrow
Z(\Gamma^\Galinv)\times Z(\gamma^\Galinv)\times\std{\m''^\Galinv}\sdp\sigma^\vee\simeq\std{\n;\sigma^\vee}.
\]
Thus $\Pi$ is irreducible by Corollary \ref{cor: stdirredcrit}.

By Corollary \ref{cor: length2} and Lemma \ref{lem: nonineq11}, respectively, we have
\[
Z(\Delta_1^\vee\setminus\Delta_k^\Galinv)\otimes\tau\le\comultG(\Pi)\text{ for some }\tau\in\IrrG
\]
while
\[
Z(\Delta_1^\vee\setminus\Delta_k^\Galinv)\otimes\tau\not\le\comultG(Z(\m^\Galinv)\sdp\sigma^\vee)=\comultG(Z(\m^\vee)\sdp\sigma^\vee)\text{ for any }\tau\in\IrrG.
\]
Hence,
\[
\Pi\not\le Z(\m^\vee)\sdp\sigma^\vee.
\]
Together with \eqref{eq: 1st embed} and \eqref{eq: es2} we conclude that
\[
Z(\m^\vee)\sdp\sigma^\vee\hookrightarrow Z(\Delta_1^\Galinv+\m'^\vee)\sdp\sigma^\vee\hookrightarrow
Z(\Delta_1^\Galinv)\times Z(\m'^\vee)\sdp\sigma^\vee
\]
which by the induction hypothesis is isomorphic to
\[
Z(\Delta_1^\Galinv)\times Z(\m'^\Galinv)\sdp\sigma^\vee\hookrightarrow\std{\m^\Galinv;\sigma^\vee}
\]
as required.
\end{proof}

\begin{remark}
Lemma \ref{lem: nonegirr} does not hold for a positive multisegment in general.
For instance, we can take $\m=\rho+\rho$ where $\tilde\rho=\lshft\rho$. (Such $\rho\notin\cuspred_\sigma$ always exists.)
Clearly, $Z(\m)\sdp\sigma$ is reducible since $\rho\times\tilde\rho\sdp\sigma$ is reducible.
\end{remark}

Using Lemma \ref{lem: nonegirr} and \cite[Proposition 6.15]{MR3573961} we can infer from Lemma \ref{lem: SIstages}
and Corollary \ref{cor: irrconseq} the following.

\begin{corollary}
Suppose that $\m$ is a ladder and $\supp\m\cap\cuspred_\sigma=\emptyset$. Let $\pi=Z(\m)$. Then
\begin{enumerate}
\item $\pi\sdp\sigma$ is \SI.
\item If $\soc(\pi_+\times\tilde\pi_+)=Z(\n)$ then $\soc(\pi\sdp\sigma)=Z(\n;Z(\m_{=0})\sdp\sigma)$.
\item The conclusion of Corollary \ref{cor: irrconseq} holds.
\end{enumerate}
\end{corollary}

\begin{remark}
Suppose that $\m=\Delta_1+\dots+\Delta_t$ and $\m'=\Delta_1'+\dots+\Delta_{t'}'$ are two ladders and let $\pi=Z(\m)$, $\pi'=Z(\m')$.
Then there is a simple combinatorial procedure to determine the multisegment corresponding to $\soc(\pi\times\pi')$ \cite[Corollary 6.16]{MR3573961}.
In particular (\cite[Proposition 6.20 and Lemma 6.21]{MR3573961}), $\soc(\pi\times\pi')\not\simeq Z(\m+\m')$
if and only if there exist integers $i,j\geq 0$ and $ \ell\geq 1$ satisfying $i+\ell\leq t,j+\ell\leq t'$, such that
\begin{enumerate}
\item $\Delta_{i+1} \prec \Delta'_{j+1},\Delta_{i+2} \prec \Delta'_{j+2},\dots,\Delta_{i+\ell} \prec \Delta'_{j+\ell}.$
\item Either $i=0$ or ${\overset{\leftarrow}{\Delta}}_i\not\prec \Delta'_{j+1}$.
\item Either $j+\ell = t'$ or ${\overset{\leftarrow}{ \Delta}}_{i+\ell}\not\prec\Delta'_{j+\ell+1}$.
\end{enumerate}
Recall that $\pi\times\pi'$ is irreducible if and only if $\soc(\pi\times\pi')\simeq Z(\m+\m')\simeq \soc(\pi'\times\pi)$.
\end{remark}

\subsection{}
In this subsection we fix $\alpha\in\cuspred_\sigma$ with $\expo(\alpha)\ge0$. We will consider $\pi=Z(\m)$ such that
\begin{equation} \label{eq: supppi2}
\supp\pi\subset\alpha[\Z]\text{ and }\supp\pi\cap\cuspred_\sigma=\emptyset.
\end{equation}
We can uniquely write $\m=\m_{>\alpha}+\m_{[\tilde\alpha,\alpha]}+\m_{<\alpha}$ where
\begin{gather*}
\supp\m_{>\alpha}\subset\{\rho\in\CuspGL:\rho>\alpha\},\\
\supp\m_{<\tilde\alpha}\subset\{\rho\in\CuspGL:\rho<\tilde\alpha\},\\
\supp\m_{[\tilde\alpha,\alpha]}\subset\{\rho\in\CuspGL:\tilde\alpha<\rho<\alpha\}.
\end{gather*}
Correspondingly,
\begin{equation} \label{eq: Z(m)product}
\pi=\pi_{>\alpha}\times\pi_{[\tilde\alpha,\alpha]}\times\pi_{<\tilde\alpha}\text{ where }
\pi_{>\alpha}=Z(\m_{>\alpha}), \pi_{[\tilde\alpha,\alpha]}=Z(\m_{[\alpha,\alpha]}),\ \pi_{<\tilde\alpha}=Z(\m_{<\tilde\alpha}).
\end{equation}
Clearly
\begin{equation} \label{eq: >alphatilde}
\tilde\pi_{>\alpha}=\widetilde{\pi_{<\tilde\alpha}},\ \ \tilde\pi_{<\alpha}=\widetilde{\pi_{>\tilde\alpha}},\ \
\tilde\pi_{[\tilde\alpha,\alpha]}=\widetilde{\pi_{[\tilde\alpha,\alpha]}}.
\end{equation}

\begin{lemma} \label{lem: >alphairred}
Suppose that $\pi\in\IrrGL$ satisfies \eqref{eq: supppi2} and in addition, $\m=\m_{>\alpha}$
(i.e., $\rho>\alpha$ for any $\rho\in\supp\pi$). Then $Z(\m)\sdp\sigma$ is irreducible.
\end{lemma}

\begin{proof}
Clearly $Z(\m)\sdp\sigma\hookrightarrow\std{\m;\sigma}$.
We prove that $Z(\m^\vee)\sdp\sigma^\vee\hookrightarrow\std{\m^\Galinv;\sigma^\vee}$ by induction on $\deg\m$.
For the induction step, write $\m=\m'+\Delta$ where $\expo(\Delta)\le\expo(\Delta')$ for any $\Delta'\le\m'$.
Then
\[
Z(\m^\vee)\sdp\sigma^\vee\hookrightarrow Z(\Delta^\vee)\times Z((\m')^\vee)\sdp\sigma^\vee
\]
which by induction hypothesis is
\[
Z(\Delta^\vee)\times Z(\m'^\Galinv)\sdp\sigma^\vee.
\]
Now $Z(\Delta^\vee)\times Z(\m'^\Galinv)$ is irreducible since $Z(\Delta^\vee)\times Z(\Delta'^\Galinv)$ is irreducible for any $\Delta'\le\m$
by the condition on $\m$. Therefore
\[
Z(\m^\vee)\sdp\sigma^\vee\hookrightarrow Z(\m'^\Galinv)\times Z(\Delta^\vee)\sdp\sigma^\vee=
Z(\m'^\Galinv)\times Z(\Delta^\Galinv)\sdp\sigma^\vee\hookrightarrow\std{\m^\Galinv}\sdp\sigma^\vee
\]
as required. Thus, the lemma follows from Corollary \ref{cor: stdirredcrit}.
\end{proof}

\begin{lemma}
Suppose that $\pi=Z(\m)\in\IrrGL$ satisfies \eqref{eq: supppi2}. Then $\pi\sdp\sigma$ is irreducible if and only if
$\pi_{[\tilde\alpha,\alpha]}\sdp\sigma$ and $\pi_{>\alpha}\times\tilde\pi_{>\alpha}$ are irreducible.
\end{lemma}

\begin{proof}
The ``only if'' part follows from \eqref{eq: Z(m)product} and \eqref{eq: indtildeinv}.
On the other hand, by \eqref{eq: >alphatilde} and Lemma \ref{lem: >alphairred} we have
\[
\pi\sdp\sigma=\pi_{>\alpha}\times\pi_{[\tilde\alpha,\alpha]}\times\pi_{<\alpha}\sdp\sigma=
\pi_{>\alpha}\times\pi_{[\tilde\alpha,\alpha]}\times\tilde\pi_{>\alpha}\sdp\sigma=
\pi_{>\alpha}\times\tilde\pi_{>\alpha}\times\pi_{[\tilde\alpha,\alpha]}\sdp\sigma.
\]
Thus, if $\pi_{[\tilde\alpha,\alpha]}\sdp\sigma$ and $\pi_{>\alpha}\times\tilde\pi_{>\alpha}$ are irreducible we get
\[
\pi\sdp\sigma=Z(\m_{>\alpha}+\tilde\m_{>\alpha})\sdp Z(\m_{[\tilde\alpha,\alpha]};\sigma)\hookrightarrow
\std{\m_{>\alpha}+\tilde\m_{>\alpha}+\pstv\m_{[\tilde\alpha,\alpha]};\sigma}=\std{\pstv\m;\sigma}.
\]
Similarly, we have
\[
\pi^\vee\sdp\sigma^\vee\hookrightarrow Z((\pstv\m)^\Galinv;\sigma^\vee).
\]
Thus, $\pi\sdp\sigma$ is irreducible by Corollary \ref{cor: stdirredcrit}.
\end{proof}

\begin{corollary} \label{cor: genirr}
Suppose that $\alpha\in\{\tilde\alpha,\tilde\alpha[1],\tilde\alpha[2]\}$ and $\pi\in\IrrGL$ satisfies \eqref{eq: supppi2}.
Then $\pi\sdp\sigma$ is irreducible if and only if $\pi_+\times\tilde\pi_+$ is irreducible.
\end{corollary}

Indeed, note that in this case $\pi_+=\pi_{>\alpha}$ (and similarly for $\tilde\pi$).
Moreover, $\pi_{[\tilde\alpha,\alpha]}\ne\one$ only if $\beta:=\lshft\alpha$ satisfies $\tilde\beta=\beta$, in which case
$\pi_{[\tilde\alpha,\alpha]}=\beta^{\times l}$ for some $l>0$.

\begin{remark}
By the results of Shahidi \cite{MR1070599}, the assumption on $\alpha$ is always satisfied if $G$ is quasi-split and $\sigma$ is generic.
For an arbitrary $\sigma\in\CuspG$, we have $\alpha\in\{\tilde\alpha,\tilde\alpha[\pm1]\}$ for all but finitely many
$\alpha\in\cuspred_\sigma$.
\end{remark}

\begin{remark}
Suppose that $\alpha=\tilde\alpha[m]$ with $m>2$.
Let $\Delta=[\tilde\alpha[2],\lshft\alpha]\in\Seg_{>0}$, so that $\tilde\Delta=\lshft\Delta\prec\Delta$.
Then $Z(\Delta+\Delta)\sdp\sigma$ is reducible, since $Z(\Delta)\times Z(\tilde\Delta)\sdp\sigma$ is reducible.
Thus, the assumption on $\alpha$ in Corollary \ref{cor: genirr} is essential.
\end{remark}


\def\cprime{$'$}
\providecommand{\bysame}{\leavevmode\hbox to3em{\hrulefill}\thinspace}
\providecommand{\MR}{\relax\ifhmode\unskip\space\fi MR }
\providecommand{\MRhref}[2]{%
  \href{http://www.ams.org/mathscinet-getitem?mr=#1}{#2}
}
\providecommand{\href}[2]{#2}

\end{document}